\documentclass[12pt, twoside, leqno]{article}

\usepackage{amssymb}
\usepackage{amsmath}
\usepackage{amsthm}
\usepackage{color}
\usepackage{mathrsfs}
\usepackage{esint}

\allowdisplaybreaks

\newtheorem{theorem}{Theorem}[section]
\newtheorem{lemma}{Lemma}[section]
\newtheorem{corollary}{Corollary}[section]
\newtheorem{proposition}{Proposition}[section]

\theoremstyle{definition}
\newtheorem{remark}{Remark}[section]
\newtheorem{definition}{Definition}[section]
\numberwithin{equation}{section}

\begin{document}
\arraycolsep=1pt
\title{A sharp Trudinger-Moser type inequality involving $L^{n}$ norm in  the entire space $\mathbb{R}^{n}$\footnotetext{\hspace{-0.35cm} 2000 {\it Mathematics Subject
Classification}. Primary 35J50; Secondary 35J20,46E35.
\endgraf {\it Key words and phrases: Trudinger-Moser inequality;
Blow up analysis; Extremal functions; unbounded domain}.
\endgraf The first author was partially supported by a US NSF grant and a Simons fellowship from Simons foundation  and and the second author was
partially supported by Natural Science Foundation of China (11601190), Natural Science Foundation of
Jiangsu Province (BK20160483) and Jiangsu University Foundation Grant (16JDG043).}}
\author{Guozhen Lu\\Department of Mathematics\\
University of Connecticut\\
Storrs, CT 06269, USA\\
E-mail: guozhen.lu@uconn.edu \vspace{0.5cm}\\ Maochun Zhu\\Faculty of Science\\Jiangsu University\\Zhenjiang, 212013, China\\E-mail: zhumaochun2006@126.com }
\date{} \maketitle

\begin{abstract}
Let $W^{1,n}\left(
\mathbb{R}
^{n}\right)  $ be the standard Sobolev space and $\left\Vert \cdot\right\Vert
_{n}$ be the $L^{n}$ norm on $\mathbb{R}^n$. We establish a sharp form of the following
Trudinger-Moser inequality involving the $L^{n}$ norm
\[
\underset{\left\Vert u\right\Vert _{W^{1,n}\left(
\mathbb{R}
^{n}\right)  }=1}{\sup}\int_{%
\mathbb{R}
^{n}}\Phi\left(  \alpha_{n}\left\vert u\right\vert ^{\frac{n}{n-1}}\left(
1+\alpha\left\Vert u\right\Vert _{n}^{n}\right)  ^{\frac{1}{n-1}}\right)
dx<+\infty
\]
in the entire space $\mathbb{R}^n$ for any $0\leq\alpha<1$, where $\Phi\left(  t\right)  =e^{t}-\underset{j=0}{\overset{n-2}{\sum}}%
\frac{t^{j}}{j!}$, $\alpha_{n}=n\omega_{n-1}^{\frac{1}{n-1}}$ and $\omega_{n-1}$ is the $n-1$ dimensional surface measure of the unit ball
in $\mathbb{R}^n$. We also show that  the above supremum is infinity for all
$\alpha\geq1$. Moreover, we prove the supremum is attained, namely, there exists a maximizer for the above supremum when $\alpha>0$ is
sufficiently small. The proof is based on the method of blow-up analysis of the nonlinear Euler-Lagrange equations of the Trudinger-Moser functionals. 

 Our result sharpens  the recent work   \cite{J. M. do1}
in which they show that the above inequality holds in a weaker form
when $\Phi(t)$ is replaced by a strictly smaller $\Phi^*(t)=e^{t}-\underset{j=0}{\overset{n-1}{\sum}}%
\frac{t^{j}}{j!}$. (Note that $\Phi(t)=\Phi^*(t)+\frac{t^{n-1}}{(n-1)!}$).

\end{abstract}

\section{Introduction}

Let $\Omega\subseteq%
\mathbb{R}
^{n}$ be an open set and $W^{1,q}_0\left(  \Omega\right)  $ be the usual Sobolev
space, that is, the completion of $C_{0}^{\infty
}\left(  \Omega\right)  $ under the norm%
\[
\left\Vert u\right\Vert _{W^{1,q}\left(  \Omega\right)  }=\left(  \int
_{\Omega}\left(  \left\vert u\right\vert ^{q}+\left\vert \nabla u\right\vert
^{q}\right)  dx\right)  ^{\frac{1}{q}}.
\]
If $1\leq q<n$, the classical Sobolev embedding says that
$W_{0}^{1,q}\left( \Omega\right)  \hookrightarrow L^{s}\left(
\Omega\right)  $ for $1\leq s\leq q^{\ast}$, where
$q^{\ast}:=\frac{nq}{n-q}$. When $q=n$, it is known that
\[
W_{0}^{1,n}\left(  \Omega\right)  \hookrightarrow L^{s}\left(  \Omega\right)
\text{ for any }n\leq s<+\infty\text{, }%
\]
but $\ W_{0}^{1,n}\left(  \Omega\right)  \varsubsetneq L^{\infty}\left(
\Omega\right)  $. The analogue of the optimal Sobolev embedding is the well-known
Trudinger-Moser inequality (\cite{moser},\cite{Trudinger}) which states as follows%

\begin{equation}
\underset{\left\Vert \nabla u\right\Vert _{L^n\left(  \Omega\right)
}\leq 1}{\underset{u\in W_{0}^{1,n}\left(  \Omega\right)
}{\sup}}\int_{\Omega }e^{\alpha\left\vert u\right\vert
^{\frac{n}{n-1}}}dx<\infty\text{ iff
}\alpha\leq\alpha_{n}=n\omega_{n-1}^{\frac{1}{n-1}}, \label{moser-tru}%
\end{equation}
where $\omega_{n-1}$ is the $n-1$ dimensional surface measure of the unit ball
in $%
\mathbb{R}
^{n}$ and $\Omega$ is a   domain of finite measure in $%
\mathbb{R}
^{n}$.

\medskip

Due to a wide range of applications in  geometric analysis and
partial differential equations (see \cite{de Figueiredo},
\cite{Adimurthi}, \cite{lam-lu} and references therein),  numerous
generalizations, extensions and applications of the Trudinger-Moser
inequality have been given. We recall in particular the result
obtained by P.-L. Lions \cite{lions}, which says that if $\left\{
u_{k}\right\}  $ is a sequence of functions in $W_{0}^{1,n}\left(
\Omega\right)  $\ with $\left\Vert \nabla u_{k}\right\Vert _{L^n(\Omega)}=1$
such that $u_{k}\rightarrow u$ weakly in $W^{1,n}\left(
\Omega\right)  $, then for any $0<p<\left( 1-\left\Vert \nabla
u\right\Vert
_{L^n(\Omega)}^{n}\right)  ^{-1/\left(  n-1\right)  }$, one has%

\[
\underset{k}{\sup}\int_{ \Omega}e^{\alpha_{n}p\left\vert u_{k}\right\vert
^{\frac{n}{n-1}}}dx<\infty.
\]
This conclusion gives more precise information than
(\ref{moser-tru}) when $u_{k}\rightarrow u\neq0$ weakly in
$W_{0}^{1,n}\left( \Omega\right)$.  Based on the result of Lions
and the blowing up analysis method, Adimurthi and O. Druet
\cite{Adimurthi} obtained an improved Trudinger-Moser type
inequality in $\mathbb{R}
^{2}$ on bounded domains $\Omega$, which can be described as follows \[
\underset{\left\Vert \nabla u\right\Vert _{2}\leq1}{\underset{u\in W_{0}%
^{1,2}\left(  \Omega\right)  }{\sup}}\int_{%
\mathbb{R}
^{2}}e^{  4\pi\left\vert u\right\vert ^{2}\left(
1+\alpha\left\Vert u\right\Vert _{2}^{2}\right)}
dx<\infty,\text{iff }\alpha<\underset{ u\in W_{0}^{1,2}\left(
\Omega\right) ,u\neq0}{\inf}\frac{\left\Vert \nabla u\right\Vert
_{2}^{2}}{\left\Vert u\right\Vert _{2}^{2}}.
\]  Subsequently, this result was extended to $L^p$ norm in two dimension and high dimension as well in  Yang \cite{yang}, Lu and Yang \cite{lu-yang}, \cite{lu-yang 1} and Zhu \cite{zhu}.

\medskip

Another interesting extension of (\ref{moser-tru}) is to construct
Trudinger-Moser inequalities for unbounded domains. In fact, we note
that, even in the case $\alpha<\alpha_{n}$, the supremum in
(\ref{moser-tru}) becomes
infinite for domains $\Omega\subseteq%
\mathbb{R}
^{n}$ with $\left\vert \Omega\right\vert =+\infty$. Related
inequalities for unbounded domains have been first considered by
D.M. Cao \cite{cao} in the case $N=2$ and for any dimension by J.M.
do \'{O} \cite{J.M. do1} and Adachi-Tanaka \cite{Adachi-Tanaka} in the subcritical case,  that is $\alpha<\alpha_{n}$.
In \cite{ruf}, B. Ruf showed that in the case $N=2$, the exponent
$\alpha_{2}=4\pi$ becomes admissible if the Dirichlet norm
$\int_{\Omega}\left\vert \nabla u\right\vert ^{2}dx$ is replaced by
Sobolev norm $\int_{\Omega}\left(  \left\vert u\right\vert
^{2}+\left\vert \nabla u\right\vert ^{2}\right)  dx$, more
precisely, he proved that

\begin{equation}
\underset{\int_{\mathbb{R}
^{2}}\left(  \left\vert u\right\vert ^{2}+\left\vert \nabla
u\right\vert ^{2}\right)  dx\leq1}{\underset{u\in W^{1,2}\left(
\mathbb{R}
^{2}\right)  }{\sup}}\int_{%
\mathbb{R}
^{2}}\Phi\left(  \alpha\left\vert u\right\vert ^{2}\right)  dx<
+\infty,\text{ iff }\alpha\leq4\pi, %
\end{equation}
where $\Phi\left(  t\right)  =e^{t}-1$. Later, Y. X. Li and B. Ruf
\cite{liruf} extended Ruf's result to arbitrary dimension.

\medskip

Recently, M. de Souza and J. M. do \'{O} \cite{do1} obtained an Adimurthi-Druet type result in $\mathbb{R}^2$ for some weighted Sobolev space  $$E=\left\{  u\in W^{1,2}\left(
\mathbb{R}
^{2}\right)  :\int_{%
\mathbb{R}
^{2}}V\left(  x\right)  u^{2}dx<\infty\right\},$$ where the  potential $V$ is  radially  symmetric, increasing and coercive.

\medskip

In this
paper, we will  try to remove the potential $V$ in \cite{do1}, and we obtain  an Adimurthi-Druet type result for $W^{1,n}\left(
\mathbb{R}
^{n}\right)$.  Our main results read as follows
\begin{theorem}
\label{moser-trudinger} For any $0\leq\alpha<1$, the following holds:

\begin{equation}
\underset{\left\Vert u\right\Vert _{W^{1,n}\left(
\mathbb{R}
^{n}\right)  }=1}{\sup}\int_{%
\mathbb{R}
^{n}}\Phi\left(  \alpha_{n}\left\vert u\right\vert
^{\frac{n}{n-1}}\left( 1+\alpha\left\Vert u\right\Vert
_{n}^{n}\right)  ^{\frac{1}{n-1}}\right) dx<\infty,
\label{moser-Trudi}\end{equation}
where $\Phi\left(  t\right)  =e^{t}-\underset{j=0}{\overset{n-2}{\sum}}%
\frac{t^{j}}{j!}$. Moreover,   for any $\alpha\geq1,$ the
supremum is infinite.
\end{theorem}

At this point, we call attention to the recent work of M. de Souza and J. M. do \'{O} in \cite{J. M. do1}, where the authors establish an analogue of (\ref{moser-Trudi}) under the additional assumption that  $\Phi\left(  t\right)$  is substituted by a smaller function $\Psi \left(  t\right)  =e^{t}-\underset{j=0}{\overset{n-1}{\sum}}%
\frac{t^{j}}{j!}$. But, they did not address whether the supremum is finite  when $\alpha=1$. Here, we remark that by using the test function sequence constructed in Section 2, we can show that the supreme in (\ref{moser-Trudi}) is infinity when $\alpha=1$. Therefore, our results indeed  improve substantially  the result in \cite{J. M. do1}.

We set
\[
S=\underset{\left\Vert u\right\Vert _{W^{1,n}\left(
\mathbb{R}
^{n}\right)  }=1}{\sup}\int_{%
\mathbb{R}
^{n}}\Phi\left(  \alpha_{n}\left\vert u\right\vert ^{\frac{n}{n-1}}\left(
1+\alpha\left\Vert u\right\Vert _{n}^{n}\right)  ^{\frac{1}{n-1}}\right)  dx.
\]
The existence of an extremal function for the above supremum is only known when $\alpha=0$ as shown in \cite{liruf}.
However, whether an  extremal function for the above supremum exists or not is
not known for $\alpha>0$. Our next aim is to show that the supremum above is
attained when $\alpha$ is chosen small enough, that is

\begin{theorem}
\label{attain}There exists $u_{\alpha}\in W^{1,n}\left(
\mathbb{R}
^{n}\right)  $ with $\left\Vert u_{\alpha}\right\Vert _{W^{1,n}\left(
\mathbb{R}
^{n}\right)  }=1$ such that
\[
S=\int_{%
\mathbb{R}
^{n}}\Phi\left(  \alpha_{n}\left\vert u_{\alpha}\right\vert ^{\frac{n}{n-1}%
}\left(  1+\alpha\left\Vert u_{\alpha}\right\Vert _{n}^{n}\right)  ^{\frac
{1}{n-1}}\right)  dx
\]
for sufficiently small $\alpha$.
\end{theorem}

The first result about existence of the extremal function for
Trudinger-Moser inequality was given by L. Carleson and S.Y.A. Chang
in \cite{c-c}, where it is proved that the supremum in
(\ref{moser-tru}) indeed has
extremals by using symmetrization argument, when $\Omega$ is a ball in $%
\mathbb{R}
^{n}$. This actually brings a surprise, since it is well-known that the
Sobolev-inequality has no extremals on any finite domain $\Omega\neq$\ $%
\mathbb{R}
^{n}$. Later, M. Flucher \cite{Flucher} showed that this result continues to
hold for any smooth domain in $%
\mathbb{R}
^{2}$ and Lin in \cite{lin} generalized the result to any dimension.
More existence results can be found in several papers, see
e.g. Y.X. Li \cite{Li2} and \cite{Li 1} for Trudinger-Moser inequalities on
compact Riemannian manifold, \cite{ruf} and \cite{liruf} for on
unbounded domains in $\mathbb{R}^n$, and Lu and Yang \cite{lu-yang},\cite{lu-yang 1} and
Zhu \cite{zhu} for Trudinger-Moser inequalities involving a remainder
term. For the existence of critical points for the supercritical regime, i.e. the Trudinger-Moser energy functionals constrained to manifold $M=  \left\{  {u\in W^{1,n}_0(\Omega), \left\Vert \nabla u\right\Vert _{L^n((\Omega))}>1}\right\}$, see del Pino,  Musso and  Ruf  \cite{de} and Malchiodi and Martinazzi \cite{Malchiodi}, and references therein.

\medskip

We now sketch the idea of proving Theorem \ref{moser-trudinger} and
Theorem \ref{attain}.

\medskip

1. The proof of the second part of Theorem \ref{moser-trudinger} is
based on a test function argument. Unlike in the case for  bounded domains
\cite{yang}, we cannot construct the test function by the
eigenfunction of the first eigenvalue problem: $$\underset {u\in
W_{0}^{1,n}\left(  \Omega\right)  ,u\neq0}{\inf}\frac{\left\Vert
\nabla u\right\Vert _{n}^{n}}{\left\Vert u\right\Vert _{n}^{n}},$$
since the above infimum is actually not attained when $\Omega=\mathbb{R}
^{n}$.  To overcome this difficulty, we
will construct a new test function sequence (see Section 2
for more details).

\medskip

2. For the proof of the first part of Theorem
\ref{moser-trudinger}, we will  carry out the standard blowing up
analysis procedure. This method is based on a blowing up analysis of
sequences of solutions to $n$-Laplacian in $\mathbb{R}
^{n}$ with exponential growth, and it has been successfully applied in
the proof of the Trudinger-Moser inequalities and related existence
results in bounded domains (see
\cite{Adimurthi},\cite{zhu},\cite{lu-yang} and \cite{lu-yang 1}). In the unbounded case, one will encounter many new difficulties. For
instance, when the blowing up phenomenon arises, a crucial step is to show the strong
convergence of $u_k$ in $L^n$ norm ($u_k$ are the  maximizers for a sequence of subcritical Trudinger-Moser energy functionals). We recall that in  \cite{J. M. do1},  the authors proved the strong convergence under the additional assumption that  $\Phi\left(  t\right)$  is substituted by a smaller function $\Psi \left(  t\right)  =e^{t}-\underset{j=0}{\overset{n-1}{\sum}}%
\frac{t^{j}}{j!}$. In our case, we remove that unnatural assumption, and in order to prove the strong convergence, we will need more careful analysis and different technique (see $\S 4.1$ for more details)

\medskip

 3. To prove Theorem \ref{attain},
we will adapt ideas in the spirit of    proofs given in, e.g.,  Y. X. Li \cite{Li2, Li 1} and Y. X. Li and B. Ruf in
\cite{liruf}. We first derive the upper bound for the
Trudinger-Moser inequality from a result of L. Carleson and S.Y.A.
Chang \cite{c-c} when the blowing up arises, and then construct a
function sequence to show that the upper bound can actually be
surpassed.

\medskip

This paper is organized as follows. In Section 2 we prove the sharpness
of the inequality in Theorem \ref{moser-trudinger} by constructing a
appropriate test function sequence; Section 3 is devoted to proving the
existence of radially symmetric maximizing sequence for the critical
functional; in Section 4, we apply the blowing up analysis to
analyze the asymptotic behavior of the maximizing sequence near and
far away from the origin, and give the proof for the first part of
Theorem \ref{moser-trudinger}; in Section 5, we prove the existence
result---Theorem \ref{attain} by constructing a test function
sequence.

\medskip

Throughout this paper, the letter $c$ denotes a constant which may
vary from line to line.

\section{\bigskip The test functions argument}

In this section, we prove the sharpness of the inequality in Theorem
\ref{moser-trudinger}. Namely, we will show that if $\alpha\geq1$, then the sumpremum is infinity.

\begin{proof}
[Proof of the Second Part of Theorem \ref{moser-trudinger}]
Setting
\[
u_{k}=\frac{1}{\omega_{n-1}^{\frac{1}{n}}}\left\{
\begin{array}
[c]{c}%
\left(  \log k\right)  ^{-\frac{1}{n}}\ \log\frac{R_{k}}{\left\vert
x\right\vert }\text{ \ \ \ \ }\frac{R_{k}}{k}<\left\vert x\right\vert \leq
R_{k}\\
\left(  \log k\right)  ^{\frac{n-1}{n}}\ \ \text{\ \ \ \ \ \ \ \ \ \ \ }%
0<\left\vert x\right\vert \leq\frac{R_{k}}{k}%
\end{array}
\right.
\]
where $R_{k}:=\frac{\left(  \log k\right)  ^{1/2n}}{\log\log k}\rightarrow
\infty$, as $k\rightarrow\infty$. We can easily verify that
\[
\int_{%
\mathbb{R}
^{n}}\left\vert \nabla u_{k}\right\vert ^{n}dx=1
\]
and%

\begin{align*}
\left\Vert u_{k}\right\Vert _{n}^{n} &  =\left(  \int_{B_{R_{k}/k}}%
+\int_{B_{R_{k}}\backslash B_{R_{k}/k}}\right)  \left\vert u_{k}\right\vert
^{n}dx\\
&  =\frac{\left(  \log k\right)  ^{n-1}}{n}\left(  \frac{R_{k}}{k}\right)
^{n}+\frac{R_{k}^{n}}{\log k}\int_{\frac{1}{k}}^{1}\left(  \log r\right)
^{n}r^{n-1}dr\\
&  =\frac{C_{n}R_{k}^{n}}{\log k}\left(  1+o\left(  1\right)  \right)
\rightarrow0\text{ as }k\rightarrow\infty,
\end{align*}
where $C_{n}=\int_{\frac{1}{k}}^{1}\left(  \log r\right)  ^{n}r^{n-1}dr$.
Therefore we have%

\[
\left\Vert u_{k}\right\Vert _{W^{1,n}\left(
\mathbb{R}
^{n}\right)  }^{n}=1+\frac{C_{n}R_{k}^{n}}{\log k}\left(  1+o\left(  1\right)
\right)
\]
$\ $

Since%
\[
1+\frac{\left\Vert u_{k}\right\Vert _{n}^{n}}{\left\Vert u_{k}\right\Vert
_{W^{1,n}}^{n}}=\frac{1+2\left\Vert u_{k}\right\Vert _{n}^{n}}{1+\left\Vert
u_{k}\right\Vert _{n}^{n}},
\]
then on the ball $B_{R_{k}/k}$, we have
\begin{align*}
&  \alpha_{n}\frac{\left\vert u_{k}\right\vert ^{\frac{n}{n-1}}}{\left\Vert
u_{k}\right\Vert _{W^{1,n}\left(
\mathbb{R}
^{n}\right)  }^{\frac{n}{n-1}}}\left(  1+\frac{\left\Vert u_{k}\right\Vert
_{n}^{n}}{\left\Vert u_{k}\right\Vert _{W^{1,n}\left(
\mathbb{R}
^{n}\right)  }^{n}}\right)  ^{\frac{1}{n-1}}\\
&  =n\omega_{n-1}^{\frac{1}{n-1}}\left\vert u_{k}\right\vert ^{\frac{n}{n-1}%
}\frac{\left(  1+2\left\Vert u_{k}\right\Vert _{n}^{n}\right)  ^{\frac{1}%
{n-1}}}{\left(  1+\left\Vert u_{k}\right\Vert _{n}^{n}\right)  ^{\frac{2}%
{n-1}}}\\
&  =n\log k\left(  1-\frac{1}{n-1}\left\Vert u_{k}\right\Vert _{n}^{2n}%
+\frac{2}{n-1}\left\Vert u_{k}\right\Vert _{n}^{3n}\left(  1+o\left(
1\right)  \right)  \right)
\end{align*}
and $\left\vert B_{R_{k}/k}\right\vert =\frac{\omega_{n-1}}{n}\exp\left(
n\log R_{k}-n\log k\right)  $.

\bigskip Thus%
\begin{align*}
&  \underset{\left\Vert u\right\Vert _{W^{1,n}\left(
\mathbb{R}
^{n}\right)  }=1}{\sup}\int_{%
\mathbb{R}
^{n}}\Phi\left(  \alpha_{n}\left\vert u\right\vert ^{\frac{n}{n-1}}\left(
1+\left\Vert u\right\Vert _{n}^{n}\right)  \right)  dx\\
&  \geq c\int_{B_{R_{k}/k}}\exp\left(  \alpha_{n}\frac{\left\vert
u_{k}\right\vert ^{\frac{n}{n-1}}}{\left\Vert u_{k}\right\Vert _{W^{1,n}%
}^{\frac{n}{n-1}}}\left(  1+\frac{\left\Vert u_{k}\right\Vert _{n}^{n}%
}{\left\Vert u_{k}\right\Vert _{W^{1,n}}^{n}}\right)  \right)  dx\\
&  \geq c\exp\left(  n\log k\left(  1-\frac{1}{n-1}\left\Vert u_{k}\right\Vert
_{n}^{2n}+\frac{2}{n-1}\left\Vert u_{k}\right\Vert _{n}^{3n}\left(  1+o\left(
1\right)  \right)  \right)  +n\log R_{k}-n\log k\right)  \\
&  =c\exp\left(  n\log k\left(  -\frac{1}{n-1}\left\Vert u_{k}\right\Vert
_{n}^{2n}+\frac{2}{n-1}\left\Vert u_{k}\right\Vert _{n}^{3n}\left(  1+o\left(
1\right)  \right)  \right)  +n\log R_{k}\right)
\end{align*}

 Because
\[
n\log R_{k}=n\log\left(  \frac{\left(  \log k\right)  ^{1/2n}}{\log\log
k}\right)  =\frac{1}{2}\log\log k-n\log\log\log k
\]
and$\ $
\begin{align*}
&  n\log k\left(  -\frac{1}{n-1}\left\Vert u_{k}\right\Vert _{n}^{2n}+\frac
{2}{n-1}\left\Vert u_{k}\right\Vert _{n}^{3n}\left(  1+o\left(  1\right)
\right)  \right)  \\
&  =\frac{-n}{n-1}\frac{C_{n}^{2}R_{k}^{2n}}{\log k}\left(  1+o\left(
1\right)  \right)  \\
&  =\frac{-n}{n-1}C_{n}^{2}\frac{1}{\left(  \log\log k\right)  ^{2n}}\left(
1+o\left(  1\right)  \right)  ,
\end{align*}
we can get
\begin{align*}
&  \int_{%
\mathbb{R}
^{n}}\Phi\left(  \alpha_{n}\left\vert u_{k}\right\vert ^{\frac{n}{n-1}}\left(
1+\left\Vert u_{k}\right\Vert _{n}^{n}\right)  \right)  dx\\
&  \geq c\exp\left(  n\log k\left(  -\left\Vert u_{k}\right\Vert _{n}%
^{2n}\left(  1+o\left(  1\right)  \right)  \right)  +n\log R_{k}\right)  \\
&  =c\exp\left(  \frac{1}{2}\log\log k-n\log\log\log k-\frac{nC_{n}^{2}}%
{n-1}\frac{1}{\left(  \log\log k\right)  ^{2n}}\left(  1+o\left(  1\right)
\right)  \right)  \\
&  \rightarrow\infty.
\end{align*}
The proof is finished.
\end{proof}

\section{The maximizing sequence for critical functional}

We first present a technical lemma contributed by Jo\~{a}o Marco do \'{O}, et
al \cite{J. M. do}.

\begin{lemma}
\label{jmdo}Let $\left\{  u_{k}\right\}  $ be a sequence in $W^{1,n}\left(
\mathbb{R}
^{n}\right)  $ such that $\left\Vert u_{k}\right\Vert _{W^{1,n}}=1$ and
$u_{k}\rightarrow u\neq0$, weakly in $W^{1,n}\left(
\mathbb{R}
^{n}\right)  $. If
\[
0<q<q_{n}\left(  u\right)  :=\frac{1}{\left(  1-\left\Vert u\right\Vert
_{W^{1,n}}^{n}\right)  ^{1/\left(  n-1\right)  }},
\]
then%
\[
\underset{k}{\sup}\int_{%
\mathbb{R}
^{n}}\Phi\left(  \alpha_{n}q\left\vert u_{k}\right\vert ^{\frac{n}{n-1}%
}\right)  dx<\infty.
\]

\end{lemma}

Let $\left\{  R_{k}\right\}  $ be an increasing sequence which diverges to
infinity, and $\left\{  \beta_{k}\right\}  $ an increasing sequence which
converges to $\alpha_{n}$. Setting
\[
I_{\beta_{k}}^{\alpha}\left(  u\right)  =\int_{B_{R_{k}}}\Phi\left(  \beta
_{k}\left\vert u\right\vert ^{\frac{n}{n-1}}\left(  1+\alpha\left\Vert
u\right\Vert _{{n}}^{{n}}\right)  ^{\frac{1}{n-1}}\right)  dx
\]
and
\[
H=\left\{  \left.  u\in W_{0}^{1,n}\left(  B_{R_{k}}\right)  \right\vert
\left\Vert u\right\Vert _{W^{1,n}}=1\right\}  .
\]

We have

\begin{lemma}
For any $0\leq\alpha \leq 1$, there exists an extremal function $u_{k}\in H$
such that%
\[
I_{\beta_{k}}^{\alpha}\left(  u_{k}\right)  =\underset{u\in H}{\sup}%
I_{\beta_{k}}^{\alpha}\left(  u\right)
\]

\end{lemma}

\begin{proof}
There exists a sequence of $\left\{  v_{i}\right\}  \in H$ such that
\[
\underset{i\rightarrow\infty}{\lim}I_{\beta_{k}}^{\alpha}\left(  v_{i}\right)
=\underset{u\in H}{\sup}I_{\beta_{k}}^{\alpha}\left(  u\right)  .
\]
Since $v_{i}$ is bounded in $W^{1,n}\left(
\mathbb{R}
^{n}\right)  $, there exists a subsequence which will still be denoted by $v_{i}$,
such that%
\[%
\begin{array}
[c]{c}%
v_{i}\rightarrow u_{k}\text{ weakly in }W^{1,n}\left(
\mathbb{R}
^{n}\right)  ,\\
v_{i}\rightarrow u_{k}\text{ strongly in }L^{s}\left(
B_{R_{k}}\right) ,
\end{array}
\]
for any $1<s<\infty$ as $i\rightarrow\infty$. Hence $v_{i}\rightarrow u_{k}$
a.e. in $%
\mathbb{R}
^{n}$, and%
\begin{align*}
g_{i} &  =\Phi\left\{  \beta_{k}\left\vert v_{i}\right\vert ^{\frac{n}{n-1}%
}\left(  1+\alpha\left\Vert v_{i}\right\Vert _{{n}}^{{n}}\right)  ^{\frac
{1}{n-1}}\right\}  \\
&  \rightarrow g_{k}=\Phi\left\{  \beta_{k}\left\vert u_{k}\right\vert
^{\frac{n}{n-1}}\left(  1+\alpha\left\Vert u_{k}\right\Vert _{{n}}^{{n}%
}\right)  ^{\frac{1}{n-1}}\right\}
\end{align*}
a.e. in $%
\mathbb{R}
^{n}$. We claim that $u_{k}\neq0$. If not, $1+\alpha\left\Vert v_{i}%
\right\Vert _{{n}}^{{n}}\rightarrow1$, and then $g_{i}$ is bounded in
$L^{r}\left(  B_{R_{k}}\right)  $ for some $r>1$, thus $g_{i}\rightarrow0$.
Therefore, $\underset{u\in H}{\sup}I_{\beta_{k}}^{\alpha}\left(  u\right)
=0$, which is impossible. \ By Lemma \ref{jmdo}, we have\ for any
$q<q_{n}\left(  u_{k}\right)  :=\frac{1}{\left(  1-\left\Vert u_{k}\right\Vert
_{W^{1,n}}\right)  ^{1/\left(  n-1\right)  }}$,%
\[
\underset{i\rightarrow\infty}{\lim\sup}\int_{%
\mathbb{R}
^{n}}\Phi\left(  \alpha_{n}q\left\vert v_{i}\right\vert ^{\frac{n}{n-1}%
}\right)  dx<\infty.
\]
Since $\alpha \leq 1$, we have%
\[
1+\alpha\left\Vert u_{k}\right\Vert _{{n}}^{{n}}<1+\left\Vert u_{k}\right\Vert
_{W^{1,n}}^{n}<\frac{1}{1-\left\Vert u_{k}\right\Vert _{W^{1,n}}^{n}}%
=q_{n}\left(  u_{k}\right)  ,
\]
then $g_{i}$ is bounded in $L^{s}$ for some $s>1$, and $g_{i}\rightarrow
g_{k}$ strongly in $L^{1}\left(  B_{R_{k}}\right)  $, as $i\rightarrow\infty
$\emph{. }Therefore, the extremal function is attained for the case $\beta
_{k}<\alpha_{n}$ and $\left\Vert u_{k}\right\Vert _{W^{1,n}}=1$.
\end{proof}
Similar as in \cite{liruf}, we give the following
\begin{lemma}
Let $u_{k}$ be as above. Then

(i) $u_{k}$ is a maximizing sequence for $S;$

(ii) $u_{k}$ may be chosen to be radially symmetric and decreasing.
\end{lemma}

\begin{proof}
(i) Let $\eta$ be a cut-off function which is $1$ on $B_{1}$ and $0$ on $%
\mathbb{R}
^{n}\backslash B_{2}$. Then given any $\varphi\in W^{1,n}\left(
\mathbb{R}
^{n}\right)  $ with $\int_{%
\mathbb{R}
^{n}}\left(  \left\vert \varphi\right\vert ^{n}+\left\vert \nabla
\varphi\right\vert ^{n}\right)  dx=1$, we have%
\[
\tau^{n}\left(  L\right)  :=\int_{%
\mathbb{R}
^{n}}\left(  \left\vert \nabla\eta\left(  \frac{x}{L}\right)  \varphi
\right\vert ^{n}+\left\vert \eta\left(  \frac{x}{L}\right)  \varphi\right\vert
^{n}\right)  dx\rightarrow1,\text{as }L\rightarrow+\infty.
\]
Hence for a fixed $L$ and $R_{k}>2L$,
\begin{align*}
&  \int_{B_{L}}\Phi\left(  \beta_{k}\left\vert \frac{\varphi}{\tau\left(
L\right)  }\right\vert ^{\frac{n}{n-1}}\left(  1+\alpha\left\Vert \frac
{\eta\left(  \frac{x}{L}\right)  \varphi}{\tau\left(  L\right)  }\right\Vert
_{{n}}^{{n}}\right)  ^{\frac{1}{n-1}}\right)  dx\\
&  \leq\int_{B_{2L}}\Phi\left(  \beta_{k}\left\vert \frac{\eta\left(  \frac
{x}{L}\right)  \varphi}{\tau\left(  L\right)  }\right\vert ^{\frac{n}{n-1}%
}\left(  1+\alpha\left\Vert \frac{\eta\left(  \frac{x}{L}\right)  \varphi
}{\tau\left(  L\right)  }\right\Vert _{{n}}^{{n}}\right)  ^{\frac{1}{n-1}%
}\right)  dx\\
&  \leq\int_{B_{R_{k}}}\Phi\left(  \beta_{k}\left\vert u_{k}\right\vert
^{\frac{n}{n-1}}\left(  1+\alpha\left\Vert u_{k}\right\Vert _{{n}}^{{n}%
}\right)  ^{\frac{1}{n-1}}\right)  dx.
\end{align*}
By the Levi Lemma, we have%
\begin{align*}
&  \int_{B_{L}}\Phi\left(  \alpha_{n}\left\vert \frac{\varphi}{\tau\left(
L\right)  }\right\vert ^{\frac{n}{n-1}}\left(  1+\alpha\left\Vert \frac
{\eta\left(  \frac{x}{L}\right)  \varphi}{\tau\left(  L\right)  }\right\Vert
_{{n}}^{{n}}\right)  ^{\frac{1}{n-1}}\right)  dx\\
&  \leq\underset{k\rightarrow\infty}{\lim}\int_{%
\mathbb{R}
^{n}}\Phi\left(  \beta_{k}\left\vert u_{k}\right\vert ^{\frac{n}{n-1}}\left(
1+\alpha\left\Vert u_{k}\right\Vert _{{n}}^{{n}}\right)  ^{\frac{1}{n-1}%
}\right)  dx.
\end{align*}
Letting \ $L\rightarrow\infty$, we get
\[
\int_{%
\mathbb{R}
^{n}}\Phi\left(  \alpha_{n}\left\vert \varphi\right\vert ^{\frac{n}{n-1}%
}\left(  1+\alpha\left\Vert \varphi\right\Vert _{{n}}^{{n}}\right)  ^{\frac
{1}{n-1}}\right)  dx\leq\underset{k\rightarrow\infty}{\lim}\int_{%
\mathbb{R}
^{n}}\Phi\left(  \beta_{k}\left\vert u_{k}\right\vert ^{\frac{n}{n-1}}\left(
1+\alpha\left\Vert u_{k}\right\Vert _{{n}}^{{n}}\right)  ^{\frac{1}{n-1}%
}\right)  dx.
\]
Hence,%
\begin{align*}
&  \underset{k\rightarrow\infty}{\lim}\int_{B_{R_{k}}}\Phi\left(  \beta
_{k}\left\vert u_{k}\right\vert ^{\frac{n}{n-1}}\left(  1+\alpha\left\Vert
u_{k}\right\Vert _{{n}}^{{n}}\right)  ^{\frac{1}{n-1}}\right)  dx\\
&  =\underset{\left\Vert u\right\Vert _{W^{1,n}\left(
\mathbb{R}
^{n}\right)  }=1}{\sup}\int_{%
\mathbb{R}
^{n}}\Phi\left(  \alpha_{n}\left\vert u\right\vert ^{\frac{n}{n-1}}\left(
1+\alpha\left\Vert u\right\Vert _{{n}}^{{n}}\right)  ^{\frac{1}{n-1}}\right)
dx.
\end{align*}

(ii) Let $u_{k}^{\ast}$ be the radial rearrangement of $u_{k}$. Then
\[
\tau_{k}^{n}:=\int_{%
\mathbb{R}
^{n}}\left(  \left\vert \nabla u_{k}^{\ast}\right\vert ^{n}+\left\vert
u_{k}^{\ast}\right\vert ^{n}\right)  dx\leq\int_{%
\mathbb{R}
^{n}}\left(  \left\vert \nabla u_{k}\right\vert ^{n}+\left\vert u_{k}%
\right\vert ^{n}\right)  dx=1,
\]
thus
\[
\int_{B_{R_{k}}}\Phi\left(  \beta_{k}\left\vert \frac{u_{k}^{\ast}}{\tau_{k}%
}\right\vert ^{\frac{n}{n-1}}\left(  1+\alpha\left\Vert \frac{u_{k}^{\ast}%
}{\tau_{k}}\right\Vert _{{n}}^{{n}}\right)  ^{\frac{1}{n-1}}\right)
dx\geq\int_{B_{R_{k}}}\Phi\left(  \beta_{k}\left\vert u_{k}^{\ast}\right\vert
^{\frac{n}{n-1}}\left(  1+\alpha\left\Vert u_{k}^{\ast}\right\Vert _{{n}}%
^{{n}}\right)  ^{\frac{1}{n-1}}\right)  dx.
\]
Since%
\[
\int_{B_{R_{k}}}\Phi\left(  \beta_{k}\left\vert u_{k}^{\ast}\right\vert
^{\frac{n}{n-1}}\left(  1+\alpha\left\Vert u_{k}^{\ast}\right\Vert _{{n}}%
^{{n}}\right)  ^{\frac{1}{n-1}}\right)  dx=\int_{B_{R_{k}}}\Phi\left(
\beta_{k}\left\vert u_{k}\right\vert ^{\frac{n}{n-1}}\left(
1+\alpha\left\Vert u_{k}\right\Vert _{{n}}^{{n}}\right)  ^{\frac
{1}{n-1}}\right)  dx,
\]
we have $\tau_{k}=1$. It is well-known that $\tau_{k}=1$ iff $u_{k}$ is
radial. Therefore
\begin{align*}
&  \int_{B_{R_{k}}}\Phi\left(  \beta_{k}\left\vert u_{k}^{\ast}\right\vert
^{\frac{n}{n-1}}\left(  1+\alpha\left\Vert u_{k}^{\ast}\right\Vert _{{n}}%
^{{n}}\right)  ^{\frac{1}{n-1}}\right)  dx\\
&  =\underset{\left\Vert u\right\Vert _{W^{1,n}\left(  B_{R_{k}}\right)  }%
=1}{\sup}\int_{B_{R_{k}}}\exp\left\{  \beta_{k}\left\vert u\right\vert
^{\frac{n}{n-1}}\left(  1+\alpha\left\Vert u\right\Vert _{{n}}^{{n}}\right)
^{\frac{1}{n-1}}\right\}  dx.
\end{align*}
So, we can assume $u_{k}=u_{k}\left(  \left\vert x\right\vert \right)  $, and
$u_{k}\left(  r\right)  $ is decreasing.
\end{proof}

\section{Blow up analysis}

In this section, the method of blow-up analysis will be used to analyze the
asymptotic behavior of the maximizing sequence $\left\{  u_{k}\right\}  $, and
the first part of Theorem \ref{moser-trudinger} will be finished.

\medskip

After a direct computation, the Euler-Lagrange equation for the extremal function
$u_{k}\in W_{0}^{1,n}\left(  B_{R_{k}}\right)  $ of $I_{\beta_{k}}^{\alpha
}\left(  u\right)  $ can be written as%

\begin{equation}
-\triangle_{n}u_{k}+u_{k}^{n-1}=\mu_{k}\lambda_{k}^{-1}u_{k}^{\frac{1}{n-1}%
}\Phi^{\prime}\left\{  \alpha_{k}u_{k}^{\frac{n}{n-1}}\right\}  +\gamma
_{k}u_{k}^{{n}-1} \label{equation}%
\end{equation}
where
\[
\left\{
\begin{array}
[c]{c}%
u_{k}\in W_{0}^{1,n}\left(  B_{R_{k}}\right)  ,\left\Vert u_{k}\right\Vert
_{W^{1,n}}=1,\\
\alpha_{k}=\beta_{k}\left(  1+\alpha\left\Vert u_{k}\right\Vert _{{n}}^{{n}%
}\right)  ^{\frac{1}{n-1}},\\
\mu_{k}=\left(  1+\alpha\left\Vert u_{k}\right\Vert _{{n}}^{{n}}\right)
/\left(  1+2\alpha\left\Vert u_{k}\right\Vert _{{n}}^{{n}}\right)  ,\\
\gamma_{k}=\alpha/\left(  1+2\alpha\left\Vert u_{k}\right\Vert _{{n}}^{{n}%
}\right)  ,\\
\lambda_{k}=\int_{B_{R_{k}}}u_{k}^{\frac{n}{n-1}}\Phi^{\prime}\left(
\alpha_{k}u_{k}^{\frac{n}{n-1}}\right)  .
\end{array}
\right.
\]

In the following, we denote $c_{k}=\max u_{k}=u_{k}\left(  0\right)  $. First,
we give the following important observation.

\begin{lemma}
\label{lamna}$\underset{k}{\inf}\, \,\lambda_{k}>0.$
\end{lemma}

\begin{proof}
Assume $\lambda_{k}\rightarrow0$. Then%
\begin{align}
\lambda_{k}  &  =\int_{%
\mathbb{R}
^{n}}u_{k}^{\frac{n}{n-1}}\Phi^{\prime}\left(  \alpha_{k}u_{k}^{\frac{n}{n-1}%
}\right)  dx=\int_{%
\mathbb{R}
^{n}}u_{k}^{\frac{n}{n-1}}\underset{j=n-2}{\overset{\infty}{\sum}}%
\frac{\left(  \alpha_{k}u_{k}^{\frac{n}{n-1}}\right)  ^{j}}{j!}dx\nonumber\\
&  =\int_{%
\mathbb{R}
^{n}}\left(  \frac{\alpha_{k}^{j}u_{k}^{n}}{\left(  n-2\right)  !}%
+\ldots\right)  dx\geq\frac{\alpha_{k}^{j}}{\left(  n-2\right)  !}\int_{%
\mathbb{R}
^{n}}u_{k}^{n}dx. \label{1}%
\end{align}
Since $u_{k}\left(  \left\vert x\right\vert \right)  $ is decreasing, we have
$u_{k}^{n}\left(  L\right)  \left\vert B_{L}\right\vert \leq\int_{B_{L}}%
u_{k}^{n}dx\leq1$, and then%
\begin{equation}
u_{k}^{n}\left(  L\right)  \leq\frac{n}{\omega_{n-1}L^{n}}. \label{2}%
\end{equation}
Set $\varepsilon=\frac{n}{\omega_{n-1}L^{n}}$. Then for any $x\notin B_{L}$, we
have $u_{k}\leq\varepsilon$ , and
\[
\int_{%
\mathbb{R}
^{n}\backslash B_{L}}\Phi\left(  \alpha_{k}u_{k}^{\frac{n}{n-1}}\right)
dx\leq c\int_{%
\mathbb{R}
^{n}\backslash B_{L}}u_{k}^{n}dx\leq c\lambda_{k}\rightarrow0.
\]
Since
\[
\Phi\left(  \alpha_{k}u_{k}^{\frac{n}{n-1}}\right)  =\underset{j=n-1}%
{\overset{\infty}{\sum}}\frac{\left(  \alpha_{k}u_{k}^{\frac{n}{n-1}}\right)
^{j}}{j!}\leq\underset{j=n-2}{\overset{\infty}{\sum}}\frac{\alpha_{k}%
u_{k}^{\frac{n}{n-1}}\left(  \alpha_{k}u_{k}^{\frac{n}{n-1}}\right)  ^{j}%
}{\left(  j+1\right)  j!}\leq \alpha_{k}u_{k}^{\frac{n}{n-1}}\Phi^{\prime}\left(
\alpha_{k}u_{k}^{\frac{n}{n-1}}\right)  ,
\]
we have
\begin{align*}
&  \underset{k\rightarrow\infty}{\lim}\int_{B_{L}}\Phi\left(  \alpha_{k}%
u_{k}^{\frac{n}{n-1}}\right)  dx=\underset{k\rightarrow\infty}{\lim}\left(
\int_{B_{L}\cap\left\{  u_{k}\geq1\right\}  }+\int_{B_{L}\cap\left\{
u_{k}<1\right\}  }\right)  \Phi\left(  \alpha_{k}u_{k}^{\frac{n}{n-1}}\right)
dx\\
&  \leq\underset{k\rightarrow\infty}{\lim}\left(  c\int_{B_{L}}u_{k}^{\frac
{n}{n-1}}\Phi\left(  \alpha_{k}u_{k}^{\frac{n}{n-1}}\right)  dx+\int
_{B_{L}\cap\left\{  u_{k}<1\right\}  }\Phi\left(  \alpha_{k}u_{k}^{\frac
{n}{n-1}}\right)  dx\right)  \\
&  \leq\underset{k\rightarrow\infty}{\lim}\left(  c\lambda_{k}+c\int_{B_{L}%
}u_{k}^{n}dx\right)  .
\end{align*}
By (\ref{1}), we see that $\int_{B_{L}}u_{k}^{q}dx\rightarrow0$, for any
$q>1$, and then we have

\[
\underset{k\rightarrow\infty}{\lim}\int_{B_{L}}\Phi\left(  \alpha_{k}%
u_{k}^{\frac{n}{n-1}}\right)  dx=0.
\]
This is impossible.
\end{proof}

Now, we introduce the concept of Sobolev-normalized concentrating
sequence and concentration-compactness principle as in  \cite{ruf}.

\begin{definition}
\label{SNC}A sequence $\left\{  u_{k}\right\}  \in W^{1,n}\left(
\mathbb{R}
^{n}\right)  $ is a Sobolev-normalized concentrating sequence, if

i) $\left\Vert u_{k}\right\Vert _{W^{1,n}\left(
\mathbb{R}
^{n}\right)  }=1;$

ii) $u_{k}\rightarrow0$ weakly in $W^{1,n}\left(
\mathbb{R}
^{n}\right)  ;$

iii) there exists a point $x_{0}$ such that for any $\delta>0$, $\int_{%
\mathbb{R}
^{n}\backslash B_{\delta}\left(  x_{0}\right)  }\left(  \left\vert
\nabla u_{k}\right\vert ^{n}+\left\vert u_{k}\right\vert ^{n}\right)
dx\rightarrow0$.
\end{definition}
From Lemma \ref{jmdo}, we can derive the following
\begin{lemma}
\label{lions}Let $\left\{  u_{k}\right\}  $ be a sequence satisfying
$\left\Vert u_{k}\right\Vert _{W^{1,n}\left(
\mathbb{R}
^{n}\right)  }=1$, and $u_{k}\rightarrow u$ weakly in $W^{1,n}\left(
\mathbb{R}
^{n}\right)  $. Then either $\left\{  u_{k}\right\}  $ is a
Sobolev-normalized concentrating sequence, or there exists
$\gamma>0$ such that $\Phi\left(
\left(  \alpha_{n}+\gamma\right)  \left\vert u_{k}\right\vert ^{\frac{n}{n-1}%
}\right)  $ is bounded in $L^{1}\left(
\mathbb{R}
^{n}\right)  $. \end{lemma}

\begin{lemma}
\label{attain lemma2}\bigskip If $\ \underset{k}{\sup}c_{k}<\infty$,
then Theorem \ref{moser-trudinger} and Theorem \ref{attain} hold.
\end{lemma}

\begin{proof}
For any $\varepsilon>0$, by using (\ref{2}) we can find some $L$ such that
$u_{k}\left(  x\right)  \leq\varepsilon$ when $x\notin B_{L}$. We rewrite
$\int_{%
\mathbb{R}
^{n}}\left(  \Phi\left(  \alpha_{k}u_{k}^{\frac{n}{n-1}}\right)  -\frac
{\alpha_{k}^{n-1}u_{k}^{n}}{\left(  n-1\right)  !}\right)  dx$ as%
\[
\left(  \int_{B_{L}}+\int_{%
\mathbb{R}
^{n}\backslash B_{L}}\right)  \left(  \Phi\left(  \alpha_{k}u_{k}^{\frac
{n}{n-1}}\right)  -\frac{\alpha_{k}^{n-1}u_{k}^{n}}{\left(  n-1\right)
!}\right)  dx.
\]
Since
\[
\int_{%
\mathbb{R}
^{n}\backslash B_{L}}\left(  \Phi\left(  \alpha_{k}u_{k}^{\frac{n}{n-1}%
}\right)  -\frac{\alpha_{k}^{n-1}u_{k}^{n}}{\left(  n-1\right)  !}\right)
dx=c\int_{%
\mathbb{R}
^{n}\backslash B_{L}}u_{k}^{\frac{n^{2}}{n-1}}dx\leq c\varepsilon^{\frac
{n^{2}}{n-1}-n}\int_{%
\mathbb{R}
^{n}}u_{k}^{n}dx=c\varepsilon^{\frac{n^{2}}{n-1}-n},
\]
we have
\begin{equation}
\int_{%
\mathbb{R}
^{n}}\left(  \Phi\left(  \alpha_{k}u_{k}^{\frac{n}{n-1}}\right)  -\frac
{\alpha_{k}^{n-1}u_{k}^{n}}{\left(  n-1\right)  !}\right)  dx=\int_{B_{L}%
}\left(  \Phi\left(  \alpha_{k}u_{k}^{\frac{n}{n-1}}\right)  -\frac{\alpha
_{k}^{n-1}u_{k}^{n}}{\left(  n-1\right)  !}\right)  dx+O\left(  \varepsilon
^{\frac{n^{2}}{n-1}-n}\right).  \label{33}%
\end{equation}
It follows from  $\underset{k}{\sup}c_{k}<\infty$ that
\begin{align*}
\int_{%
\mathbb{R}
^{n}}\Phi\left(  \alpha_{k}u_{k}^{\frac{n}{n-1}}\right)  dx  &  =\int_{B_{L}%
}\left(  \Phi\left(  \alpha_{k}u_{k}^{\frac{n}{n-1}}\right)  -\frac{\alpha
_{k}^{n-1}u_{k}^{n}}{\left(  n-1\right)  !}\right)  dx+\int_{%
\mathbb{R}
^{n}}\frac{\alpha_{k}^{n-1}u_{k}^{n}}{\left(  n-1\right)  !}dx+O\left(
\varepsilon^{\frac{n^{2}}{n-1}-n}\right) \\
&  \leq c\left(  L\right)  ,
\end{align*}
thus, Theorem \ref{moser-trudinger} holds. By Lemma \ref{lamna} and applying
the elliptic estimate in \cite{T} to equation (\ref{equation}), we have
$u_{k}\rightarrow u$ in $C_{loc}^{1}\left(
\mathbb{R}
^{n}\right)  $.

\medskip

When $u=0$, we claim that $\left\{  u_{k}\right\}  $ is not a
Sobolev-normalized concentrating sequence. If not, by iii) of
Definition \ref{SNC} and the fact that $\left\vert u_{k}\right\vert
$ is bounded, we have for any $\delta>0$,
\begin{align*}
\int_{%
\mathbb{R}
^{n}}u_{k}^{n}dx  & \leq\int_{B_{\delta}}u_{k}^{n}dx+\int_{%
\mathbb{R}
^{n}\backslash B_{\delta}}u_{k}^{n}dx\\
& \leq c\delta^{n}+o_{k}\left(  1\right)  .
\end{align*}
Letting $\delta\rightarrow0$, we have $\int_{%
\mathbb{R}
^{n}}u_{k}^{n}dx\rightarrow0$, as $k\rightarrow\infty$. For any $\varepsilon>0$, when $L$ is large enough, we have\ by (\ref{33}) that%
\begin{align*}
&  S+o_{k}\left(  1\right)  =\int_{%
\mathbb{R}
^{n}}\Phi\left(  \alpha_{k}u_{k}^{\frac{n}{n-1}}\right)  dx\\
&  =\int_{%
\mathbb{R}
^{n}}\frac{\alpha_{k}^{n-1}u_{k}^{n}}{\left(  n-1\right)  !}dx+\int_{B_{L}%
}\left(  \Phi\left(
\alpha_{k}\cdot
u^{\frac{n}{n-1}}\right)
-\int_{B_{L}}\frac{\alpha_{k}^{n-1}\cdot u^{n}}{\left(  n-1\right)  !}\right)
dx+O\left( \varepsilon^{\frac{n^{2}}{n-1}-n}\right)  ,
\end{align*}
then
\[
S\leq\int_{%
\mathbb{R}
^{n}}\frac{\alpha_{k}^{n-1}u_{k}^{n}}{\left(  n-1\right)
!}dx\rightarrow0,
\]
which is impossible, and thus the claim is proved.

\medskip

By Lemma \ref{lions}, we have $\int_{%
\mathbb{R}
^{n}}\Phi\left(  \alpha_{k}u_{k}^{\frac{n}{n-1}}\right)  dx\rightarrow\int_{%
\mathbb{R}
^{n}}\Phi\left(  \alpha_{n}u^{\frac{n}{n-1}}\right)  dx=0$, which is
still   impossible. Therefore, $u\neq0$.

\medskip

 Now, we show that $\int_{%
\mathbb{R}
^{n}}u_{k}^{n}\rightarrow\int_{%
\mathbb{R}
^{n}}u^{n}$. By (\ref{33}), we have%
\begin{align}
S  & =\underset{k\rightarrow\infty}{\lim}\int_{%
\mathbb{R}
^{n}}\Phi\left(  \alpha_{k}u_{k}^{\frac{n}{n-1}}\right)  dx\nonumber\\
& =\int_{%
\mathbb{R}
^{n}}\left(  \Phi\left(  \underset{k\rightarrow\infty}{\lim}\alpha_{k}%
u^{\frac{n}{n-1}}\right)  \right)  dx+\underset{k\rightarrow\infty}{\lim}%
\int_{%
\mathbb{R}
^{n}}\frac{\underset{k\rightarrow\infty}{\lim}\alpha_{k}^{n-1}\left(
u_{k}^{n}-u^{n}\right)  }{\left(  n-1\right)  !}dx.\label{4}%
\end{align}
Set
\[
\tau^{n}_k=\frac{\int_{%
\mathbb{R}
^{n}}u_{k}^{n}}{\int_{%
\mathbb{R}
^{n}}u^{n}}.
\]
By the Levi Lemma, we have $\tau_k\geq1$. \ Let $\tilde{u}=u\left(  \frac
{x}{\tau_k}\right)  $. Then, we have
\[
\int_{%
\mathbb{R}
^{n}}\left\vert \nabla\tilde{u}\right\vert ^{n}dx=\int_{%
\mathbb{R}
^{n}}\left\vert \nabla u\right\vert ^{n}dx\leq\int_{%
\mathbb{R}
^{n}}\left\vert \nabla u_{k}\right\vert ^{n}dx
\]
and
\[
\int_{%
\mathbb{R}
^{n}}\left\vert \tilde{u}\right\vert ^{n}dx=\tau_k^{n}\int_{%
\mathbb{R}
^{n}}\left\vert u\right\vert ^{n}dx\leq\int_{%
\mathbb{R}
^{n}}\left\vert u_{k}\right\vert ^{n}dx.
\]
Therefore%
\[
\int_{%
\mathbb{R}
^{n}}\left(  \left\vert \nabla\tilde{u}\right\vert ^{n}+\left\vert \tilde
{u}\right\vert ^{n}\right)  dx\leq1.
\]
Hence, we have by (\ref{4}) that%
\begin{align*}
S &  \geq\int_{%
\mathbb{R}
^{n}}\Phi\left(  \alpha_{n}\left(  1+\alpha\left\Vert \tilde{u}\right\Vert
_{n}^{n}\right)  ^{\frac{1}{n-1}}\tilde{u}^{\frac{n}{n-1}}\right)  dx\\
&  =\tau^{n}\int_{%
\mathbb{R}
^{n}}\Phi\left(  \alpha_{n}\left(  1+\alpha\tau^{n}\left\Vert u\right\Vert
_{{n}}^{{n}}\right)  ^{\frac{1}{n-1}}u^{\frac{n}{n-1}}\right)  dx\\
&  \geq\tau^{n}\int_{%
\mathbb{R}
^{n}}\Phi\left(  \underset{k\rightarrow\infty}{\lim}\alpha_{k}u^{\frac{n}%
{n-1}}\right)  dx+o(1)\\
&  =\int_{%
\mathbb{R}
^{n}}\left(  \Phi\left(  \underset{k\rightarrow\infty}{\lim}\alpha_{k}%
u^{\frac{n}{n-1}}\right)  +\left(  \tau^{n}-1\right)  \int_{%
\mathbb{R}
^{n}}\frac{\underset{k\rightarrow\infty}{\lim}\alpha_{k}^{n-1}u^{n}}{\left(
n-1\right)  !}\right)  dx+\\
&  +\left(  \tau^{n}-1\right)  \int_{%
\mathbb{R}
^{n}}\left(  \Phi\left(  \underset{k\rightarrow\infty}{\lim}\alpha_{k}%
u^{\frac{n}{n-1}}\right)  -\int_{%
\mathbb{R}
^{n}}\frac{\underset{k\rightarrow\infty}{\lim}\alpha_{k}^{n-1}u^{n}}{\left(
n-1\right)  !}\right)  +o(1)\\
&  =\left(  \tau^{n}-1\right)  \left(  \int_{%
\mathbb{R}
^{n}}\Phi\left(  \underset{k\rightarrow\infty}{\lim}\alpha_{k}u^{\frac{n}%
{n-1}}\right)  -\int_{%
\mathbb{R}
^{n}}\frac{\underset{k\rightarrow\infty}{\lim}\alpha_{k}^{n-1}u^{n}}{\left(
n-1\right)  !}\right)  +\\
&  +\underset{k\rightarrow\infty}{\lim}\int_{%
\mathbb{R}
^{n}}\Phi\left(  \alpha_{k}u_{k}^{\frac{n}{n-1}}\right)  dx+o(1)\\
&  =S+\left(  \tau^{n}-1\right)  \int_{%
\mathbb{R}
^{n}}\left(  \Phi\left(  \underset{k\rightarrow\infty}{\lim}\alpha_{k}%
u^{\frac{n}{n-1}}\right)  -\int_{%
\mathbb{R}
^{n}}\frac{\underset{k\rightarrow\infty}{\lim}\alpha_{k}^{n-1}u^{n}}{\left(
n-1\right)  !}\right)  dx+o(1)
\end{align*}
Since $\Phi\left(  \underset{k\rightarrow\infty}{\lim}\alpha_{k}u^{\frac
{n}{n-1}}\right)  -\int_{%
\mathbb{R}
^{n}}\frac{\underset{k\rightarrow\infty}{\lim}\alpha_{k}^{n-1}u^{n}}{\left(
n-1\right)  !}>0$, we have $\tau=1$, then
\[
\underset{k}{\lim}\int_{%
\mathbb{R}
^{n}}\Phi\left(  \alpha_{k}u_{k}^{\frac{n}{n-1}}\right)  dx=\int_{%
\mathbb{R}
^{n}}\Phi\left(  \alpha_{n}\left(  1+\alpha\left\Vert u\right\Vert _{{n}}%
^{{n}}\right)  ^{\frac{1}{n-1}}u^{\frac{n}{n-1}}\right)  dx
\]
Thus, $u$ is an extremal function.
\end{proof}

\bigskip

In the following, we assume $c_{k}\rightarrow+\infty\ $and perform a blow-up procedure.

\subsection{\bigskip The asymptotic behavior of $u_{k}$}

In this subsection, we investigate the asymptotic behavior of $u_{k}$. First,
we introduce the following important quatity%

\[
r_{k}^{n}=\frac{\lambda_{k}}{\mu_{k}c_{k}^{\frac{n}{n-1}}e^{\alpha_{k}%
c_{k}^{\frac{n}{n-1}}}}.
\]
By (\ref{2}), we can find a sufficiently large $L$ such that $u_{k}\leq1$ on $%
\mathbb{R}
^{n}\backslash B_{L}\,$. Then $\left(  u_{k}-u_{k}\left(  L\right)  \right)
^{+}\in W_{0}^{1,n}\left(  B_{L}\right)  $ and%

\[
\int_{B_{L}}\left\vert \nabla\left(  u_{k}-u_{k}\left(  L\right)  \right)
^{+}\right\vert ^{n}dx\leq1,
\]
hence by \cite[Theorem 1.1]{yang}, we have%

\[
\int_{B_{L}}e^{\alpha_{n}\left(  1+\beta\left\Vert u_{k}-u_{k}\left(
L\right)  \right\Vert _{n}^{n}\right)  ^{\frac{1}{n-1}}\left(  u_{k}%
-u_{k}\left(  L\right)  \right)  ^{\frac{n}{n-1}}}dx\leq c\left(  L\right)  ,
\]
provided
\[
\beta<\underset{u\in W_{0}^{1,n}\left(  B_{L}\right)  }{\inf}\frac{\left\Vert
\nabla u\right\Vert _{n}^{n}}{\left\Vert u\right\Vert _{n}^{n}}.
\]
For any $q<\alpha_{n}\left(  1+\beta\left\Vert u_{k}-u_{k}\left(  L\right)
\right\Vert _{n}^{n}\right)  ^{\frac{1}{n-1}}$, we can find a constant
$c\left(  q\right)  \,\ $such that%

\[
qu_{k}^{\frac{n}{n-1}}\leq\alpha_{n}\left(  1+\beta\left\Vert u_{k}%
-u_{k}\left(  L\right)  \right\Vert _{n}^{n}\right)  ^{\frac{1}{n-1}}\left(
\left(  u_{k}-u_{k}\left(  L\right)  \right)  ^{+}\right)  ^{\frac{n}{n-1}%
}+c\left(  q\right)  ,
\]
and then we have%

\begin{equation}
\int_{B_{L}}e^{qu_{k}^{\frac{n}{n-1}}}dx\leq c\left(  L,q\right)  .
\label{add7}%
\end{equation}
Now we take some $0<A<1$ such that
\[
\left(  1-A\right)  \beta_{k}\left(  1+\alpha\left\Vert u_{k}\right\Vert
_{{n}}^{{n}}\right)  ^{\frac{1}{n-1}}<\alpha_{n}\left(  1+\beta\left\Vert
u_{k}-u_{k}\left(  L\right)  \right\Vert _{n}^{n}\right)  ^{\frac{1}{n-1}}.
\]
Then
\begin{align*}
&  \lambda_{k}e^{-A\beta_{k}\left(  1+\alpha\left\Vert u_{k}\right\Vert _{{n}%
}^{{n}}\right)  ^{\frac{1}{n-1}}c_{k}^{\frac{n}{n-1}}}\\
&  =e^{-A\beta_{k}\left(  1+\alpha\left\Vert u_{k}\right\Vert _{{n}}^{{n}%
}\right)  ^{\frac{1}{n-1}}c_{k}^{\frac{n}{n-1}}}\left[  \left(  \int_{%
\mathbb{R}
^{n}\backslash B_{L}}+\int_{B_{L}}\right)  u_{k}^{\frac{n}{n-1}}\Phi^{\prime
}\left(  \alpha_{k}u_{k}^{\frac{n}{n-1}}\right)  dx\right]  \\
&  \leq ce^{-A\beta_{k}\left(  1+\alpha\left\Vert u_{k}\right\Vert _{{n}}%
^{{n}}\right)  ^{\frac{1}{n-1}}c_{k}^{\frac{n}{n-1}}}\left(  \int_{%
\mathbb{R}
^{n}\backslash B_{L}}u_{k}^{n}dx+\int_{B_{L}}u_{k}^{\frac{n}{n-1}}e^{\beta
_{k}\left(  1+\alpha\left\Vert u_{k}\right\Vert _{{n}}^{{n}}\right)
^{\frac{1}{n-1}}u_{k}^{\frac{n}{n-1}}}dx\right)  \\
&  \leq c\int_{B_{L}}u_{k}^{\frac{n}{n-1}}e^{\left(  1-A\right)  \beta
_{k}\left(  1+\alpha\left\Vert u_{k}\right\Vert _{{n}}^{{n}}\right)
^{\frac{1}{n-1}}u_{k}^{\frac{n}{n-1}}}dx+o\left(  1\right)  .
\end{align*}
Since $u_{k}$ converges strongly in $L^{s}\left(  B_{L}\right)  $ for any
$s>1$, by using H\"{o}lder's inequality and (\ref{add7}), we have%

\[
\lambda_{k}\leq ce^{A\alpha_{k}c_{k}^{\frac{n}{n-1}}},
\]
hence $\ $%
\begin{equation}
\ r_{k}^{n}\leq Ce^{\left(  A-1\right)  \alpha_{k}c_{k}^{\frac{n}{n-1}}%
}=o\left(  c_{k}^{-q}\right)  , \label{rk}%
\end{equation}
for any \thinspace$q>0$.

\bigskip

Now, we set
\[
\left\{
\begin{array}
[c]{c}%
m_{k}\left(  x\right)  =u_{k}\left(  r_{k}x\right)  ,\\
\phi_{k}\left(  x\right)  =\frac{m_{k}\left(  x\right)  }{c_{k}},\\
\psi_{k}\left(  x\right)  =\frac{n}{n-1}\alpha_{k}c_{k}^{\frac{1}{n-1}}\left(
m_{k}-c_{k}\right)  ,
\end{array}
\right.
\]
where $m_{k},\phi_{k}$ and $\psi_{k}$ are defined on $\Omega_{k}:=\left\{
x\in%
\mathbb{R}
^{n}:r_{k}x\in B_{1}\right\}  $. From (\ref{equation}) and (\ref{rk}), we know
$\phi_{k}\left(  x\right)  ,\psi_{k}\left(  x\right)  $ satisfy%

\begin{align}
-\triangle_{n}\phi_{k}\left(  x\right)   &  =\frac{r_{k}^{n}}{c_{k}^{n-1}%
}\left(  \mu_{k}\lambda_{k}^{-1}m_{k}^{\frac{1}{n-1}}\Phi^{\prime}\left\{
\alpha_{k}m_{k}^{\frac{n}{n-1}}\right\}  +\left(  \gamma_{k}-1\right)
m_{k}^{{n}-1}\right) \nonumber\\
&  =\left(  \frac{1}{c_{k}^{n}}\phi_{k}^{\frac{1}{n-1}}\left(  x\right)
\Phi^{\prime}\left\{  \alpha_{k}\left(  m_{k}^{\frac{n}{n-1}}-c_{k}^{\frac
{n}{n-1}}\right)  \right\}  +o\left(  1\right)  \right)  \label{equ 1}%
\end{align}
and%

\begin{align}
-\triangle_{n}\psi_{k}\left(  x\right)   &  =\left(  \frac{n\alpha_{k}}%
{n-1}\right)  ^{n-1}c_{k}r_{k}^{n}\left(  \mu_{k}\lambda_{k}^{-1}m_{k}%
^{\frac{1}{n-1}}\Phi^{\prime}\left\{  \alpha_{k}m_{k}^{\frac{n}{n-1}}\right\}
+\left(  \gamma_{k}-1\right)  m_{k}^{n-1}\right) \nonumber\\
&  =\left(  \frac{n\alpha_{k}}{n-1}\right)  ^{n-1}\left(  \left(  \frac{m_{k}%
}{c_{k}}\right)  ^{\frac{1}{n-1}}e^{\alpha_{k}\left(  m_{k}^{\frac{n}{n-1}%
}-c_{k}^{\frac{n}{n-1}}\right)  }+o\left(  1\right)  \right)  . \label{equ 2}%
\end{align}

We analyze the limit function of $\phi_{k}$ and $\psi_{k}\left(  x\right)  $.
Since $u_{k}$ is bounded in $W^{1,n}\left(
\mathbb{R}
^{n}\right)  $, there exists a subsequence such that $u_{k}\rightarrow u$
weakly in $W^{1,n}\left(
\mathbb{R}
^{n}\right)  $. Because the right side of (\ref{equ 1}) vanishes as
$k\rightarrow\infty$, then we have $\phi_{k}\rightarrow\phi$ in $C_{loc}%
^{1}\left(
\mathbb{R}
^{n}\right)  $, as $k\rightarrow\infty$, by applying the classical eatimates
\cite{T}. Therefore,%

\[
-\triangle_{n}\phi=0\text{ in }%
\mathbb{R}
^{n}.
\]
Since $\phi_{k}\left(  0\right)  =1$, by the Lionville-type theorem, we have
$\phi\equiv1$ in $%
\mathbb{R}
^{n}$.

\medskip

Now, we investigate the asymptotic behavior of $\psi_{k}$. By (\ref{rk}) and
the fact that $\phi_{k}\left(  x\right)  \leq1$, we can rewrite (\ref{equ 2})
as
\[
-\triangle_{n}\psi_{k}\left(  x\right)  =O\left(  1\right)  .
\]
By \cite[Theorem 7]{J}, we know that $osc_{B_{L}}\psi_{k}\leq c\left(
L\right)  $ for any $L>0.$ Then from the result of \cite{T}, we have
$\left\Vert \psi_{k}\right\Vert _{C^{1,\delta}\left(  B_{L}\right)  }\leq
c\left(  L\right)  $ for some $\delta >0$. Hence $\psi_{k}$ converges in $C_{loc}^{1}\left(
B_{L}\right)  $ and $m_{k}-c_{k}\rightarrow0$ in $C_{loc}^{1}\left(
B_{L}\right)  $.

\medskip

Since
\[
m_{k}^{\frac{n}{n-1}}=c_{k}^{\frac{n}{n-1}}\left(  1+\frac{m_{k}-c_{k}}{c_{k}%
}\right)  ^{\frac{n}{n-1}}=c_{k}^{\frac{n}{n-1}}\left(  1+\frac{n}{n-1}%
\frac{m_{k}-c_{k}}{c_{k}}+O\left(  \frac{1}{c_{k}^{2}}\right)  \right)  ,
\]
we have
\begin{align}
\alpha_{k}\left(  m_{k}^{\frac{n}{n-1}}-c_{k}^{\frac{n}{n-1}}\right)   &
=\alpha_{k}c_{k}^{\frac{n}{n-1}}\left(  \frac{n}{n-1}\frac{m_{k}-c_{k}}{c_{k}%
}+O\left(  \frac{1}{c_{k}^{2}}\right)  \right) \label{5}\\
&  =\psi_{k}\left(  x\right)  +o\left(  1\right)  \rightarrow\psi\left(
x\right)  \text{ in }C_{loc}^{0},\nonumber
\end{align}
and then
\begin{equation}
-\triangle_{n}\psi=\left(  \frac{nc_{n}}{n-1}\right)  ^{n-1}\exp\left\{
\psi\left(  x\right)  \right\}  , \label{6}%
\end{equation}
where $c_{n}=\underset{k\rightarrow\infty}{\lim}\alpha_{k}=\alpha_{n}\left(
1+\alpha\underset{k\rightarrow\infty}{\lim}\left\Vert u_{k}\right\Vert _{{n}%
}^{{n}}\right)  ^{\frac{1}{n-1}}$.

\medskip

Since $\psi$ is radially symmetric and decreasing, it is easy to see that
(\ref{6}) has only one solution. We can check that
\[
\psi\left(  x\right)  =-n\log\left(  1+\frac{c_{n}}{n^{\frac{n}{n-1}}%
}\left\vert x\right\vert ^{\frac{n}{n-1}}\right)  \text{ }%
\]
and%

\begin{align}
\int_{%
\mathbb{R}
^{n}}e^{\psi\left(  x\right)  }dx  &  =\omega_{n-1}\frac{n-1}{n}\left(
\frac{n^{\frac{n}{n-1}}}{c_{n}}\right)  ^{n-1}\int_{0}^{\infty}\left(
1+t\right)  ^{-n}t^{n-2}dt\nonumber\\
&  =\omega_{n-1}\frac{n-1}{n}\left(  \frac{n^{\frac{n}{n-1}}}{c_{n}}\right)
^{n-1}\cdot\frac{1}{n-1}=\frac{1}{1+\alpha\underset{k\rightarrow\infty}{\lim
}\left\Vert u_{k}\right\Vert _{{n}}^{{n}}}. \label{add3}%
\end{align}

For any $A>1$, let $u_{k}^{A}=\min\left\{  u_{k},\frac{c_{k}}{A}\right\}  $.

\begin{lemma}
\bigskip For any $A>1$, there holds%
\[
\underset{k\rightarrow\infty}{\lim\sup}\int_{%
\mathbb{R}
^{n}}\left(  \left\vert u_{k}^{A}\right\vert ^{n}+\left\vert \nabla u_{k}%
^{A}\right\vert ^{n}\right)  dx\leq1-\frac{A-1}{A}\frac{1}{1+\alpha
\underset{k\rightarrow\infty}{\lim}\left\Vert u_{k}\right\Vert _{{n}}^{{n}}}%
\]

\end{lemma}

\begin{proof}
\bigskip Since $\left\vert \left\{  x:u_{k}\geq\frac{c_{k}}{A}\right\}
\right\vert \left\vert \frac{c_{k}}{A}\right\vert ^{n}\leq\int_{\left\{
u_{k}\geq\frac{c_{k}}{A}\right\}  }\left\vert u_{k}\right\vert ^{n}dx\leq1$,
we can find a sequence $\rho_{k}\rightarrow0$ such that
\[
\left\{  x:u_{k}\geq\frac{c_{k}}{A}\right\}  \subset B_{\rho_{k}}.
\]
Since $u_{k}$ converges in $L^{s}\left(  B_{1}\right)  $ for any $s>1$, we
have%
\[
\underset{k\rightarrow\infty}{\lim}\int_{\left\{  u_{k}\geq\frac{c_{k}}%
{A}\right\}  }\left\vert u_{k}^{A}\right\vert ^{s}dx\leq\underset
{k\rightarrow\infty}{\lim}\int_{\left\{  u_{k}\geq\frac{c_{k}}{A}\right\}
}\left\vert u_{k}\right\vert ^{s}dx=0,
\]
and then for any $s>0$,
\[
\underset{k\rightarrow\infty}{\lim}  \int_{%
\mathbb{R}
^{n}}\left(  u_{k}-\frac{c_{k}}{A}\right)  ^{+}\left\vert u_{k}\right\vert
^{s}dx=0.
\]

Testing (\ref{equation}) with $\left(  u_{k}-\frac{c_{k}}{A}\right)  ^{+}$ we have%
\begin{align*}
&  \int_{%
\mathbb{R}
^{n}}\left(  \left\vert \nabla\left(  u_{k}-\frac{c_{k}}{A}\right)
^{+}\right\vert ^{n}+\left(  u_{k}-\frac{c_{k}}{A}\right)  ^{+}\left\vert
u_{k}\right\vert ^{n-1}\right)  dx\\
&  =\int_{%
\mathbb{R}
^{n}}\left(  u_{k}-\frac{c_{k}}{A}\right)  ^{+}\mu_{k}\lambda_{k}^{-1}%
u_{k}^{\frac{1}{n-1}}\Phi^{\prime}\left\{  \alpha_{k}u_{k}^{\frac{n}{n-1}%
}\right\}  dx+o\left(  1\right)  \\
&  \geq\int_{B_{Rr_{k}}}\left(  u_{k}-\frac{c_{k}}{A}\right)  ^{+}\mu
_{k}\lambda_{k}^{-1}u_{k}^{\frac{1}{n-1}}\exp\left\{  \alpha_{k}u_{k}%
^{\frac{n}{n-1}}\right\}  dx+o\left(  1\right)  \\
&  =\int_{B_{R}}\frac{\left(  m_{k}-\frac{c_{k}}{A}\right)  }{c_{k}}\left(
\frac{m_{k}-c_{k}}{c_{k}}+1\right)  ^{\frac{1}{n-1}}\exp\left\{  \psi
_{k}\left(  x\right)  +o\left(  1\right)  \right\}  dx+o\left(  1\right)  \\
&  \geq\frac{A-1}{A}\int_{B_{R}}e^{\psi\left(  x\right)  }dx.
\end{align*}
Letting $R\rightarrow\infty,k\rightarrow\infty$, by
(\ref{add3}), we have %
\[
\underset{k\rightarrow\infty}{\lim\inf}\int_{%
\mathbb{R}
^{n}}\left(  \left\vert \nabla\left(  u_{k}-\frac{c_{k}}{A}\right)
^{+}\right\vert ^{n}+\left(  u_{k}-\frac{c_{k}}{A}\right)  ^{+}\left\vert
u_{k}\right\vert ^{n-1}\right)  dx\geq\frac{A-1}{A}\frac{1}{\left(
1+\alpha\underset{k\rightarrow\infty}{\lim}\left\Vert u_{k}\right\Vert _{{n}%
}^{{n}}\right)  }.
\]
Now,  observe that%
\begin{align*}
&  \int_{%
\mathbb{R}
^{n}}\left(  \left\vert \nabla u_{k}^{A}\right\vert ^{n}+\left\vert u_{k}%
^{A}\right\vert ^{n}\right)  dx\\
&  =1-\int_{%
\mathbb{R}
^{n}}\left(  \left\vert \nabla\left(  u_{k}-\frac{c_{k}}{A}\right)
^{+}\right\vert ^{n}+\left(  u_{k}-\frac{c_{k}}{A}\right)  ^{+}\left\vert
u_{k}\right\vert ^{n-1}\right)  dx\\
&  +\int_{%
\mathbb{R}
^{n}}\left(  u_{k}-\frac{c_{k}}{A}\right)  ^{+}\left\vert u_{k}\right\vert
^{n-1}dx-\int_{\left\{  u_{k}>\frac{c_{k}}{A}\right\}  }\left\vert
u_{k}\right\vert ^{n}dx+\int_{\left\{  u_{k}>\frac{c_{k}}{A}\right\}
}\left\vert u_{k}^{A}\right\vert ^{n}dx\\
&  \leq1-\frac{A-1}{A}\frac{1}{1+\alpha\underset{k\rightarrow\infty}{\lim
}\left\Vert u_{k}\right\Vert _{{n}}^{{n}}}+o\left(  1\right)  ,
\end{align*}
the proof is finished.
\end{proof}

\begin{lemma}
$\underset{k}{\lim}\left\Vert u_{k}\right\Vert _{{n}}^{{n}}=0$.
\end{lemma}

\begin{proof}
If $\{u_{k}\}$ is a Sobolev-normalized concentrating sequence, then $\underset{k}{\lim}\left\Vert u_{k}\right\Vert _{{n}}^{{n}}=0$.
If $\{u_{k}\}$ is not a Sobolev-normalized concentrating sequence, and $\underset{k}{\lim}\left\Vert u_{k}\right\Vert _{{n}}^{{n}}\neq0$. For  $A$ large enough, there exist
some constant $\varepsilon_{0}>0$ such that
\[
\int_{%
\mathbb{R}
^{n}}\left(  \left\vert \nabla u_{k}^{A}\right\vert ^{n}+\left\vert u_{k}%
^{A}\right\vert ^{n}\right)  dx=1-\frac{1}{  1+\left(  \alpha
+\varepsilon_{0}\right)  \underset{k}{\lim}\left\Vert u_{k}\right\Vert _{{n}%
}^{{n}}  }<1.
\]
By \cite[Theorem 1.1]{liruf}, we have $\int_{%
\mathbb{R}
^{n}}\Phi\left(  q\alpha_{n}\left\vert u_{k}^{A}\right\vert ^{\frac{n}{n-1}%
}\right)  dx\leq\infty$, for any $$q<\left(  \frac{\left(  1+\left(
\alpha+\varepsilon_{0}\right)  \underset{k}{\lim}\left\Vert u_{k}\right\Vert
_{{n}}^{{n}}\right)  }{\left(  \alpha+\varepsilon_{0}\right)  \underset
{k}{\lim}\left\Vert u_{k}\right\Vert _{{n}}^{{n}}}\right)  ^{\frac{1}{n-1}}.$$
Since $\alpha<1$, $\left\Vert u_{k}\right\Vert _{W^{1,n}}=1$ and
$\underset{k}{\lim}\left\Vert u_{k}\right\Vert _{{n}}^{{n}}\neq0$, we can take
some $\varepsilon_{0}$ such that $\left(  \alpha+\varepsilon_{0}\right)
\underset{k}{\lim}\left\Vert u_{k}\right\Vert _{{n}}^{{n}}<1$, and then
\[
\left(  1+\alpha\underset{k}{\lim}\left\Vert u_{k}\right\Vert _{{n}}^{{n}%
}\right)  ^{\frac{1}{n-1}}<\left(  \frac{1+\left(  \alpha+\varepsilon
_{0}\right)  \underset{k}{\lim}\left\Vert u_{k}\right\Vert _{{n}}^{{n}}%
}{\left(  \alpha+\varepsilon_{0}\right)  \underset{k}{\lim}\left\Vert
u_{k}\right\Vert _{{n}}^{{n}}}\right)  ^{\frac{1}{n-1}}.
\]
Therefore\
\begin{equation}
\int_{%
\mathbb{R}
^{n}}\Phi\left(  p^{\prime}\alpha_{k}\left\vert u_{k}^{A}\right\vert
^{\frac{n}{n-1}}\right)  dx<\infty,\label{add}
\end{equation}
for some $p^{\prime}>1$.

\medskip

Now, we claim that $\Delta_{n}u_{k}\in
L^{r} $ for some $r>1$.

When $\int_{\left\{  u_{k}>\frac{c_{k}}{A}\right\}  }\left\vert \nabla
u_{k}\right\vert ^{n}dx\rightarrow0$. In this case, we can easily derive the above claim by the Trudinger-Moser inequalities on bounded domains and (\ref{add}). When $\int_{\left\{  u_{k}>\frac{c_{k}}{A}\right\} }\left\vert \nabla
u_{k}\right\vert ^{n}dx\geq c$, for some $c>0$. We split $u_{k}$ as
$u_{k}^{1}+u_{k}^{2}$, with $u_{k}^{1}\rightarrow c\delta_{0}$ and
$\int_{\left\{  u_{k}>\frac{c_{k}}{A}\right\}  }\left\vert \nabla u_{k}%
^{2}\right\vert ^{n}dx\rightarrow0$.\ \ \  Since $\alpha<1$, we have
\begin{align*}
1+\alpha\left\Vert u_{k}\right\Vert _{n}^{n}  & =1+\alpha\left\Vert u_{k}%
^{2}\right\Vert _{n}^{n}+o_{k}\left(  1\right)  <\frac{1}{1-\left\Vert
u_{k}^{2}\right\Vert _{W^{1,n}}^{n}}+o_{k}\left(  1\right)  \\
& \leq\frac{1}{\left\Vert \nabla u_{k}^{1}\right\Vert _{n}^{n}}+o_{k}\left(
1\right)  \leq\frac{1}{\left\Vert \nabla u_{k}\right\Vert _{L^{n}\left(
\left\{  u_{k}>\frac{c_{k}}{A}\right\}  \right)  }^{n}}+o_{k}\left(  1\right)
,
\end{align*}
and then there exists some constant $s>1$ such that $\left(  1+\alpha\left\Vert u_{k}\right\Vert _{n}^{n}\right)  s$ $\leq\frac
{1}{\left\Vert \nabla u_{k}\right\Vert _{L^{n}\left(  \left\{  u_{k}%
>\frac{c_{k}}{A}\right\}  \right)  }^{n}}$, therefore by the classic
Trudinger-Moser inequality on the bounded domain and (\ref{add}), the claim is proved.

\medskip

Based on the the claim above and the classic elliptic estimate, we know that $u_{k}$ is
bounded near $0$, and which contradicts the assumption that $c_{k}\rightarrow\infty$. Therefore $\underset{k}{\lim}\left\Vert u_{k}\right\Vert _{{n}}^{{n}}=0$, and the lemma is proved.
\end{proof}

\begin{remark}
\label{remark}\ From the above lemma, we have
\[
\underset{k}{\lim}\alpha_{k}=\alpha_{n},\underset{k}{\lim}\mu_{k}=1,
\]%
\[
\underset{k\rightarrow\infty}{\lim\sup}\int_{%
\mathbb{R}
^{n}}\left(  \left\vert \nabla u_{k}^{A}\right\vert ^{n}+\left\vert u_{k}%
^{A}\right\vert ^{n}\right)  dx=\frac{1}{A},\,
\]%
\[
\psi\left(  x\right)  =-n\log\left(  1+\left(  \frac{\omega_{n-1}}{n}\right)
^{\frac{1}{n-1}}\left\vert x\right\vert ^{\frac{n}{n-1}}\right)  ,
\]
and
\begin{align}
&  \underset{R\rightarrow\infty}{\lim}\underset{k\rightarrow\infty}{\lim}%
\frac{1}{\lambda_{k}}\int_{B_{Rr_{k}}}u_{k}^{\frac{n}{n-1}}\exp\left(
\alpha_{k}u_{k}^{\frac{n}{n-1}}\right)  dx=\underset{R\rightarrow\infty}{\lim
}\underset{k\rightarrow\infty}{\lim}\frac{1}{\mu_{k}}\int_{B_{R}}%
e^{\psi\left(  x\right)  }dx\nonumber\\
&  =\underset{k\rightarrow\infty}{\lim}\frac{1}{\left(  1+\alpha\underset
{k}{\lim}\left\Vert u_{k}\right\Vert _{{n}}^{{n}}\right)  \mu_{k}%
}=1.\label{add9}%
\end{align}

\end{remark}

\begin{corollary}
\label{tent to 0}We have $\underset{k\rightarrow\infty}{\lim}\int_{%
\mathbb{R}
^{n}\backslash B_{\delta}}\left(  \left\vert \nabla u_{k}\right\vert
^{n}+\left\vert u_{k}\right\vert ^{n}\right)  dx=0$, for any $\delta>0$, and
then $\underset{k\rightarrow\infty}{\lim}u_{k}:\equiv0$.
\end{corollary}

\begin{lemma}
\bigskip\label{lemma 8}\ We have
\begin{equation}
\underset{k\rightarrow\infty}{\lim}\int_{%
\mathbb{R}
^{n}}\Phi\left(  \alpha_{k}u_{k}^{\frac{n}{n-1}}\right)  dx\leq\underset
{R\rightarrow\infty}{\lim}\underset{k\rightarrow\infty}{\lim}\int_{B_{Rr_{k}}%
}\left(  \exp\left(  \alpha_{k}u_{k}^{\frac{n}{n-1}}\right)  -1\right)
dx=\underset{k\rightarrow\infty}{\lim\sup}\frac{\lambda_{k}}{c_{k}^{\frac
{n}{n-1}}}, \label{7}%
\end{equation}
moreover,
\begin{equation}
\frac{\lambda_{k}}{c_{k}}\rightarrow\infty\text{ and }\underset{k}{\sup}%
\frac{c_{k}^{\frac{n}{n-1}}}{\lambda_{k}}\leq\infty. \label{7.1}%
\end{equation}
.
\end{lemma}

\begin{proof}
\bigskip For any $A>1$, from the expression of $\lambda_{k}$, we have
\begin{align*}
\int_{%
\mathbb{R}
^{n}}\Phi\left(  \alpha_{k}u_{k}^{\frac{n}{n-1}}\right)  dx &  \leq\int
_{u_{k}<\frac{c_{k}}{A}}\Phi\left(  \alpha_{k}u_{k}^{\frac{n}{n-1}}\right)
dx+\int_{u_{k}\geq\frac{c_{k}}{A}}\Phi^{\prime}\left(  \alpha_{k}u_{k}%
^{\frac{n}{n-1}}\right)  dx\\
&  \leq\int_{%
\mathbb{R}
^{n}}\Phi\left(  \alpha_{k}\left(  u_{k}^{A}\right)  ^{\frac{n}{n-1}}\right)
dx+\int_{u_{k}\geq\frac{c_{k}}{A}}\Phi^{\prime}\left(  \alpha_{k}u_{k}%
^{\frac{n}{n-1}}\right)  dx\\
&  \leq\int_{%
\mathbb{R}
^{n}}\Phi\left(  \alpha_{k}\left(  u_{k}^{A}\right)  ^{\frac{n}{n-1}}\right)
dx+\left(  \frac{A}{c_{k}}\right)  ^{\frac{n}{n-1}}\lambda_{k}\int_{u_{k}%
\geq\frac{c_{k}}{A}}\frac{u_{k}^{\frac{n}{n-1}}}{\lambda_{k}}\Phi^{\prime
}\left(  \alpha_{k}u_{k}^{\frac{n}{n-1}}\right)  dx.
\end{align*}
Thanks to Remark \ref{remark} and \cite[Theorem 1.1]{liruf}, $\Phi\left(
\alpha_{k}\left(  u_{k}^{A}\right)  ^{\frac{n}{n-1}}\right)  $ is bounded in
$L^{r}$ for some $r>1$. Since $u_{k}^{A}\rightarrow0$ a.e. in $%
\mathbb{R}
^{n}$ as $k\rightarrow\infty$, we have
\[
\int_{%
\mathbb{R}
^{n}}\Phi\left(  \alpha_{k}\left(  u_{k}^{A}\right)  ^{\frac{n}{n-1}}\right)
dx\rightarrow\int_{%
\mathbb{R}
^{n}}\Phi\left(  0\right)  dx=0\text{, as }k\rightarrow\infty\text{.}%
\]
Hence, we have by (\ref{add9}) that
\begin{align*}
\underset{k\rightarrow\infty}{\lim}\int_{%
\mathbb{R}
^{n}}\Phi\left(  \alpha_{k}u_{k}^{\frac{n}{n-1}}\right)  dx &  \leq\left(
\frac{A}{c_{k}}\right)  ^{\frac{n}{n-1}}\lambda_{k}\int_{u_{k}\geq\frac{c_{k}%
}{A}}\frac{u_{k}^{\frac{n}{n-1}}}{\lambda_{k}}\Phi^{\prime}\left(  \alpha
_{k}u_{k}^{\frac{n}{n-1}}\right)  dx+o\left(  1\right)  \\
&  =\underset{k\rightarrow\infty}{\lim}A^{\frac{n}{n-1}}\frac{\lambda_{k}%
}{c_{k}^{\frac{n}{n-1}}}+o\left(  1\right)
\end{align*}
Letting $A\rightarrow1$ and $k\rightarrow\infty$ we obtain (\ref{7}).

\medskip

If $\frac{\lambda_{k}}{c_{k}}$ is bounded or $\underset{k}{\sup}\frac
{c_{k}^{\frac{n}{n-1}}}{\lambda_{k}}=\infty$, from (\ref{7}), we have
\[
\underset{k\rightarrow\infty}{\lim}\int_{%
\mathbb{R}
^{n}}\Phi\left(  \alpha_{k}u_{k}^{\frac{n}{n-1}}\right)  dx=0,
\]
which is impossible.
\end{proof}

\begin{lemma}
\bigskip\label{dirac}For any $\varphi\in C_{0}^{\infty}\left(
\mathbb{R}
^{n}\right)  $, we have
\end{lemma}%

\[
\bigskip\int_{%
\mathbb{R}
^{n}}\varphi\mu_{k}\lambda_{k}^{-1}c_{k}u_{k}^{\frac{1}{n-1}}\Phi^{\prime
}\left(  \alpha_{k}u_{k}^{\frac{n}{n-1}}\right)  dx=\varphi\left(  0\right)
.
\]

\begin{proof}
\bigskip As \cite[Lemma 3.6]{liruf}, we split the integral as follows
\begin{align*}
\int_{%
\mathbb{R}
^{n}}\varphi\mu_{k}\lambda_{k}^{-1}c_{k}u_{k}^{\frac{1}{n-1}}\Phi^{\prime
}\left(  \alpha_{k}\left(  u_{k}\right)  ^{\frac{n}{n-1}}\right)  dx &
\leq\left(  \int_{\left\{  u_{k}\geq\frac{c_{k}}{A}\right\}  \backslash
B_{Rr_{k}}}+\int_{B_{Rr_{k}}}+\int_{\left\{  u_{k}<\frac{c_{k}}{A}\right\}
}\right)  \ldots dx\\
&  =I_{1}+I_{2}+I_{3}.
\end{align*}
Now, we have
\begin{align*}
I_{1} &  \leq A\left\Vert \varphi\right\Vert _{L^{\infty}}\int_{\left\{
u_{k}\geq\frac{c_{k}}{A}\right\}  \backslash B_{Rr_{k}}}\mu_{k}\lambda
_{k}^{-1}c_{k}u_{k}^{\frac{1}{n-1}}\Phi^{\prime}\left(  \alpha_{k}\left(
u_{k}\right)  ^{\frac{n}{n-1}}\right)  dx\\
&  \leq A\left\Vert \varphi\right\Vert _{L^{\infty}}\left(  \int_{%
\mathbb{R}
^{n}}-\int_{B_{Rr_{k}}}\right)  \mu_{k}\lambda_{k}^{-1}u_{k}^{\frac{n}{n-1}%
}\Phi^{\prime}\left(  \alpha_{k}\left(  u_{k}\right)  ^{\frac{n}{n-1}}\right)
dx\\
&  \leq A\left\Vert \varphi\right\Vert _{L^{\infty}}\left(  1-\int_{B_{R}}%
\exp\left(  \alpha_{k}m_{k}^{\frac{n}{n-1}}-\alpha_{k}c_{k}^{\frac{n}{n-1}%
}\right)  \right)  \\
&  =A\left\Vert \varphi\right\Vert _{L^{\infty}}\left(  1-\int_{B_{R}}%
\exp\left(  \psi_{k}\left(  x\right)  +o\left(  1\right)  \right)  \right)
\end{align*}
and
\begin{align*}
I_{2} &  =\int_{B_{Rr_{k}}}\varphi\mu_{k}\lambda_{k}^{-1}c_{k}u_{k}^{\frac
{1}{n-1}}\Phi^{\prime}\left(  \alpha_{k}u_{k}^{\frac{n}{n-1}}\right)  dx\\
&  =\int_{B_{R}}\varphi\left(  r_{k}x\right)  \left(  \frac{m_{k}}{c_{k}%
}\right)  ^{\frac{1}{n-1}}\exp\left(  \alpha_{k}m_{k}^{\frac{n}{n-1}}%
-\alpha_{k}c_{k}^{\frac{n}{n-1}}\right)  dx+o(1)\\
&  =\varphi\left(  0\right)  \int_{B_{R}}\exp\left(  \psi_{k}\left(  x\right)
+o\left(  1\right)  \right)  dx+o\left(  1\right)  =\varphi\left(  0\right)
+o\left(  1\right)  \rightarrow\varphi\left(  0\right)  \text{, as
}k\rightarrow\infty\text{.}%
\end{align*}
By (\ref{7.1}) and H\"{o}lder's inequality, we obtain
\begin{align*}
I_{3} &  =\int_{\left\{  u_{k}<\frac{c_{k}}{A}\right\}  }\varphi\mu_{k}%
\lambda_{k}^{-1}c_{k}u_{k}^{\frac{1}{n-1}}\Phi^{\prime}\left(  \alpha
_{k}\left(  u_{k}\right)  ^{\frac{n}{n-1}}\right)  dx\\
&  =\int_{%
\mathbb{R}
^{n}}\varphi\mu_{k}\lambda_{k}^{-1}c_{k}\left(  u_{k}^{A}\right)  ^{\frac
{1}{n-1}}\Phi^{\prime}\left(  \alpha_{k}\left(  u_{k}^{A}\right)  ^{\frac
{n}{n-1}}\right)  dx\\
&  \leq c_{k}\left\Vert \varphi\right\Vert _{L^{\infty}}\lambda_{k}%
^{-1}\left(  \int_{%
\mathbb{R}
^{n}}\left(  u_{k}^{A}\right)  ^{\frac{q}{n-1}}dx\right)  ^{\frac{1}{q}%
}\left(  \int_{%
\mathbb{R}
^{n}}\Phi^{\prime}\left(  q^{\prime}\alpha_{k}\left(  u_{k}^{A}\right)
^{\frac{n}{n-1}}\right)  dx\right)  ^{\frac{1}{q^{\prime}}}\rightarrow0,\text{
as }k\rightarrow\infty,
\end{align*}
for any $q^{\prime}<A^{\frac{1}{n-1}}$, such that $q=\frac{q^{\prime}%
}{q^{\prime}-1}$ large enough. Letting $R\rightarrow\infty$, by Remark
\ref{remark}, the lemma is proved.
\end{proof}

\begin{lemma}
\label{tend to G 1}On any $\Omega\Subset%
\mathbb{R}
^{n}\backslash\{0\}$, we have $c_{k}^{\frac{1}{n-1}}u_{k}\rightarrow G_{\alpha}\in
C^{1,\alpha}\left(  \Omega\right)  $ weakly in $W^{1,q}\left(  \Omega\right)
$ for any $1<q<n$, where $G_{\alpha}$ is a Green function satisfying
\end{lemma}
\begin{equation}
-\triangle_{n}G_{\alpha}=\delta_{0}+\left(  \alpha-1\right)  G_{\alpha}^{n-1}.
\label{121}%
\end{equation}

\begin{proof}
Setting $U_{k}=c_{k}^{\frac{1}{n-1}}u_{k}$. \ By (\ref{equation}), $U_{k}$
satisfy:%
\begin{equation}
-\triangle_{n}U_{k}=\mu_{k}c_{k}\lambda_{k}^{-1}u_{k}^{\frac{1}{n-1}}%
\Phi^{\prime}\left\{  \alpha_{k}u_{k}^{\frac{n}{n-1}}\right\}  +\left(
\gamma_{k}-1\right)  U_{k}^{n-1}. \label{add 8}%
\end{equation}

For $t\geq1$, denote$\ U_{k}^{t}=\min\left\{  U_{k},t\right\}  $ and
$\Omega_{t}^{k}=\left\{  0\leq U_{k}\leq t\right\}  $. Testing (\ref{add 8})
with $U_{k}^{t}$, we have%
\[
\int_{%
\mathbb{R}
^{n}}-U_{k}^{t}\triangle_{n}U_{k}dx+\left(  1-\gamma_{k}\right)  \int_{%
\mathbb{R}
^{n}}U_{k}^{t}U_{k}^{n-1}dx\leq\int_{%
\mathbb{R}
^{n}}U_{k}^{t}\mu_{k}c_{k}\lambda_{k}^{-1}u_{k}^{\frac{1}{n-1}}\Phi^{\prime
}\left\{  \alpha_{k}u_{k}^{\frac{n}{n-1}}\right\}  dx.
\]
Since $\gamma_{k}\rightarrow\alpha<1$, as $k\rightarrow\infty$, we have
\begin{align*}
\int_{\Omega_{t}^{k}}\left\vert \nabla U_{k}^{t}\right\vert ^{n}%
dx+\int_{\Omega_{t}^{k}}\left\vert U_{k}^{t}\right\vert ^{n}dx &  \leq\int_{%
\mathbb{R}
^{n}}\left(  -U_{k}^{t}\triangle_{n}U_{k}dx+U_{k}^{t}U_{k}^{n-1}\right)  dx\\
&  \leq c\int_{%
\mathbb{R}
^{n}}U_{k}^{t}\mu_{k}c_{k}\lambda_{k}^{-1}u_{k}^{\frac{1}{n-1}}\Phi^{\prime
}\left\{  \alpha_{k}u_{k}^{\frac{n}{n-1}}\right\}  dx\leq ct.
\end{align*}
Let $\eta$ be a radially symmetric cut-off function which is $1$ on $B_{R}$
and $0$ on $B_{2R}^{c}$, and satisfy $\left\vert \nabla\eta\right\vert \leq1$
(when $R$ large enough). Then%
\[
\int_{B_{2R}}\left\vert \nabla\eta U_{k}^{t}\right\vert ^{n}dx\leq\int
_{B_{2R}}\left\vert \nabla\eta\right\vert ^{n}\left\vert U_{k}^{t}\right\vert
^{n}dx+\int_{B_{2R}}\left\vert \eta\nabla U_{k}^{t}\right\vert ^{n}dx\leq
c_{1}\left(  R\right)  t+c_{2}\left(  R\right)  ,
\]
taking $t$ large enough, we have
\[
\int_{B_{2R}}\left\vert \nabla\eta U_{k}^{t}\right\vert ^{n}dx\leq2c\left(
R\right)  t.
\]
Then by an adaptation of an argument due to Struwe \cite{st} (also
see \cite{liruf}), we can obtain that $\left\Vert \nabla
U_{k}\right\Vert _{L^{q}\left(  B_{R}\right)  }\leq c\left(
q,n,\alpha,R\right)  $ for any $1<q<n$, and thus $\left\Vert
U_{k}\right\Vert _{L^{p}\left(  B_{R}\right) }\leq\infty$, for any
$0<p<\infty$. By Corollary \ref{tent to 0}, we know $\exp\left\{
\alpha_{k}u_{k}^{\frac{n}{n-1}}\right\}  $ is bounded in
$L^{r}\,\left(  \Omega\backslash\left\{  B_{\delta}\right\}  \right)
$ for any $r>0$ and $\delta>0$. Then applying \cite[Theorem 2.8]{J}
and the result of \cite{T}, we have $\left\Vert U_{k}\right\Vert
_{C^{1,\alpha}\left( B_{R}\right)  }\leq c$, then
$c_{k}^{\frac{1}{n-1}}u_{k}\rightarrow G_{\alpha }$ weakly in
$W^{1,q}\left(  B_{R}\right)  $. So we are done.
\end{proof}

\bigskip Next, as \cite[Lemma 3.8]{liruf}, we can obtain the following
asymptotic representation of $G_{\alpha}$, which will be used to prove the existence of Trudinger-Moser functions.

\begin{lemma}
\label{tend to G}\bigskip$G_{\alpha}\in C_{loc}^{1,\beta}\left(
\mathbb{R}
^{n}\backslash\left\{  0\right\}  \right)  $ for some $\beta>0$, and near $0$,
we have
\begin{equation}
\bigskip G_{\alpha}=-\frac{n}{\alpha_{n}}\log r+A+O\left(  r^{n}\log
^{n}r\right)  . \label{13}%
\end{equation}
Moreover, for any $\delta>0$, we have
\begin{equation}
\underset{k\rightarrow\infty}{\lim}\left(  \int_{%
\mathbb{R}
^{n}\backslash B_{\delta}}\left\vert \nabla U_{k}\right\vert ^{n}dx+\left(
1-\alpha\right)  \int_{%
\mathbb{R}
^{n}\backslash B_{\delta}}U_{k}^{n}dx\right)  =\omega_{n-1}\left\vert
G_{\alpha}^{\prime}\left(  \delta\right)  \right\vert ^{n-1}\delta^{n-1}.
\label{add 9}%
\end{equation}

\end{lemma}

\begin{proof}
The proof of (\ref{13}) is similar as \cite[Lemma 3.8]{liruf}, here we only
give the proof for (\ref{add 9}).

By Corollary \ref{tent to 0}, we have%
\begin{equation}
\int_{%
\mathbb{R}
^{n}\backslash B_{\delta}}u_{k}^{\frac{n}{n-1}}\Phi^{\prime}\left\{
\alpha_{k}u_{k}^{\frac{n}{n-1}}\right\}  dx\leq c\int_{%
\mathbb{R}
^{n}\backslash B_{\delta}}u_{k}^{n}dx\rightarrow0.\label{14}%
\end{equation}
Testing (\ref{add 8}) with $U_{k}$, we get%
\begin{align*}
&  \int_{%
\mathbb{R}
^{n}\backslash B_{\delta}}\left\vert \nabla U_{k}\right\vert ^{n}%
dx+\int_{\partial B_{\delta}}\left\vert \nabla U_{k}\right\vert ^{n-2}%
U_{k}\frac{\partial U_{k}}{\partial n}dx=-\int_{%
\mathbb{R}
^{n}\backslash B_{\delta}}div\left(  \left\vert \nabla U_{k}\right\vert
^{n-2}\nabla U_{k}\right)  U_{k}dx\\
&  =\int_{%
\mathbb{R}
^{n}\backslash B_{\delta}}\mu_{k}c_{k}^{\frac{n}{n-1}}\lambda_{k}^{-1}%
u_{k}^{\frac{n}{n-1}}\Phi^{\prime}\left\{  \alpha_{k}u_{k}^{\frac{n}{n-1}%
}\right\}  dx+\int_{%
\mathbb{R}
^{n}\backslash B_{\delta}}\left(  \gamma_{k}-1\right)  U_{k}^{n-1}U_{k}dx.
\end{align*}
By (\ref{14}), (\ref{7.1}), we have%
\begin{align*}
\underset{k\rightarrow\infty}{\lim}\int_{%
\mathbb{R}
^{n}\backslash B_{\delta}}\left\vert \nabla U_{k}\right\vert ^{n}dx &
=-\underset{k\rightarrow\infty}{\lim}\int_{\partial B_{\delta}}\left\vert
\nabla U_{k}\right\vert ^{n-2}U_{k}\frac{\partial U_{k}}{\partial n}dx+\left(
\alpha-1\right)  \underset{k\rightarrow\infty}{\lim}\int_{%
\mathbb{R}
^{n}\backslash B_{\delta}}U_{k}^{n-1}U_{k}dx\\
&  =-G_{\alpha}\left(  \delta\right)  \int_{\partial B_{\delta}}\left\vert
\nabla G_{\alpha}\right\vert ^{n-2}\frac{\partial G_{\alpha}}{\partial
n}dx+\left(  \alpha-1\right)  \underset{k\rightarrow\infty}{\lim}\int_{%
\mathbb{R}
^{n}\backslash B_{\delta}}U_{k}^{n-1}U_{k}dx\\
&  =\omega_{n-1}\left\vert G_{\alpha}^{\prime}\left(  \delta\right)
\right\vert ^{n-1}\delta^{n-1}+\left(  \alpha-1\right)  \underset
{k\rightarrow\infty}{\lim}\int_{%
\mathbb{R}
^{n}\backslash B_{\delta}}U_{k}^{n-1}U_{k}dx
\end{align*}
Thus
\[
\underset{k\rightarrow\infty}{\lim}\left(  \int_{%
\mathbb{R}
^{n}\backslash B_{\delta}}\left\vert \nabla U_{k}\right\vert ^{n}dx+\left(
1-\alpha\right)  \int_{%
\mathbb{R}
^{n}\backslash B_{\delta}}U_{k}^{n}dx\right)  =\omega_{n-1}\left\vert
G_{\alpha}^{\prime}\left(  \delta\right)  \right\vert ^{n-1}\delta^{n-1}.
\]
\ \ The proof is finished.
\end{proof}

\bigskip

\begin{proof}
[Proof for the first part of Theorem \ref{moser-trudinger}]By \bigskip
(\ref{2}), we can choose some $L>0$ such that $u_{k}\left(  L\right)  <1$, and
then
\[
\int_{%
\mathbb{R}
^{n}\backslash B_{L}}\exp\left\{  \beta_{k}\left\vert u_{k}\right\vert
^{\frac{n}{n-1}}\left(  1+\alpha\left\Vert u_{k}\right\Vert _{{n}}^{{n}%
}\right)  ^{\frac{1}{n-1}}\right\}  dx\leq c\int_{%
\mathbb{R}
^{n}\backslash B_{L}}\left\vert u_{k}\right\vert ^{n}dx\leq c.
\]
Now, we consider the case on $B_{L}$. \ Since $\left(  u_{k}-u_{k}\left(
L\right)  \right)  ^{+}\in W_{0}^{1,n}\left(  B_{L}\right)  $ and for some
$c>0$,
\begin{align*}
u_{k}^{\frac{n}{n-1}} &  =\left(  \left(  u_{k}-u_{k}\left(  L\right)
\right)  ^{+}+u_{k}\left(  L\right)  \right)  ^{\frac{n}{n-1}}\\
&  \leq\left(  \left(  u_{k}-u_{k}\left(  L\right)  \right)  ^{+}\right)
^{\frac{n}{n-1}}+c\left(  \left(  u_{k}-u_{k}\left(  L\right)  \right)
^{+}\right)  ^{\frac{1}{n-1}}u_{k}\left(  L\right)  +u_{k}\left(  L\right)
^{\frac{n}{n-1}},
\end{align*}
by Lemma \ref{tend to G 1}, we know $c_{k}^{\frac{1}{n-1}}u_{k}\rightarrow
G_{\alpha}$, then $u_{k}\left(  L\right)  =\frac{G_{\alpha}\left(  L\right)
}{c_{k}^{\frac{1}{n-1}}}$. Therefore, we have
\begin{align*}
u_{k}^{\frac{n}{n-1}} &  \leq\left(  \left(  u_{k}-u_{k}\left(  L\right)
\right)  ^{+}\right)  ^{\frac{n}{n-1}}+c\left(  \frac{\left(  u_{k}%
-u_{k}\left(  L\right)  \right)  ^{+}}{c_{k}}\right)  ^{\frac{1}{n-1}}%
+u_{k}\left(  L\right)  ^{\frac{n}{n-1}}\\
&  \leq\left(  \left(  u_{k}-u_{k}\left(  L\right)  \right)  ^{+}\right)
^{\frac{n}{n-1}}+c.
\end{align*}
Thus \
\begin{align*}
&  \int_{B_{L}}\exp\left\{  \beta_{k}\left\vert u_{k}\right\vert ^{\frac
{n}{n-1}}\left(  1+\alpha\left\Vert u_{k}\right\Vert _{{n}}^{{n}}\right)
^{\frac{1}{n-1}}\right\}  dx\\
&  \leq c\int_{B_{L}}\exp\left\{  \beta_{k}\left(  \left(  u_{k}-u_{k}\left(
L\right)  \right)  ^{+}\right)  ^{\frac{n}{n-1}}\left(  1+\alpha\left\Vert
u_{k}\right\Vert _{{n}}^{{n}}\right)  ^{\frac{1}{n-1}}\right\}  dx\\
&  \leq c\int_{B_{L}}\exp\left\{  \beta_{k}\left(  \left(  u_{k}-u_{k}\left(
L\right)  \right)  ^{+}\right)  ^{\frac{n}{n-1}}\left(  \left(  1+\alpha
\left\Vert u_{k}\right\Vert _{{n}}^{{n}}\right)  ^{\frac{1}{n-1}}-1\right)
\right\}  \exp\left(  \beta_{k}\left(  \left(  u_{k}-u_{k}\left(  L\right)
\right)  ^{+}\right)  ^{\frac{n}{n-1}}\right)  dx\\
&  \leq c\exp\left\{  \beta_{k}c_{k}^{\frac{n}{n-1}}\left(  \left(
1+\alpha\left\Vert u_{k}\right\Vert _{{n}}^{{n}}\right)  ^{\frac{1}{n-1}%
}-1\right)  \right\}  \int_{B_{L}}\exp\left(  \beta_{k}\left(  \left(
u_{k}-u_{k}\left(  L\right)  \right)  ^{+}\right)  ^{\frac{n}{n-1}}\right)
dx.
\end{align*}
From Lemma \ref{tend to G 1} and Lemma \ref{tend to G}, we know $\left\Vert
c_{k}^{\frac{1}{n-1}}u_{k}\right\Vert _{{n}}$ is bounded. Recalling the fact
that $\left\Vert u_{k}\right\Vert _{{n}}^{{n}}\rightarrow0$, and applying the
classic Trudinger-Moser inequality, we have
\begin{align*}
&  \int_{B_{L}}\exp\left\{  \beta_{k}\left\vert u_{k}\right\vert ^{\frac
{n}{n-1}}\left(  1+\alpha\left\Vert u_{k}\right\Vert _{{n}}^{{n}}\right)
^{\frac{1}{n-1}}\right\}  dx\\
&  \leq c\exp\left\{  \frac{\alpha\beta_{k}c_{k}^{\frac{n}{n-1}}}%
{n-1}\left\Vert u_{k}\right\Vert _{{n}}^{{n}}\right\}  \int_{B_{L}}\exp\left(
\beta_{k}\left(  \left(  u_{k}-u_{k}\left(  L\right)  \right)  ^{+}\right)
^{\frac{n}{n-1}}\right)  dx\\
&  =c\exp\left\{  \frac{\beta_{k}\alpha}{n-1}\left\Vert c_{k}^{\frac{1}{n-1}%
}u_{k}\right\Vert _{{n}}^{{n}}\right\}  \int_{B_{L}}\exp\left(  \beta
_{k}\left(  \left(  u_{k}-u_{k}\left(  L\right)  \right)  ^{+}\right)
^{\frac{n}{n-1}}\right)  dx\\
&  \leq c.
\end{align*}

\end{proof}

\section{\bigskip Existence of the extremal function}

In this section, we will show that the existence of the extremal functions of
the Trudinger-Moser ineuqality involving $L^{n}$ norm in $%
\mathbb{R}
^{n}$. For this, we first establish the upper bound for critical
functional when $c_{k}\rightarrow\infty$, and then construct an
explicit test function, which provides a lower bound for the
supremum of our Trudinger-Moser inequality, meanwhile, this lower
bound equals to the upper bound.

\medskip

In order to prove the existence of the extremal functions, we need the
following famous result due to L. Carleson and S.Y.A. Chang \cite{c-c}, which
often plays a key role in proof of existence result (see \cite{liruf},
\cite{yang}, \cite{lu-yang 1} and \cite{zhu}).

\begin{theorem}
[Carleson and Chang]\bigskip\label{C-C} Let $B$ be a unit ball in $%
\mathbb{R}
^{n}$. Given a function sequence $\left\{  u_{k}\right\}  \subset W_{0}%
^{1,n}\left(  B\right)  $ with $\left\Vert \nabla u_{k}\right\Vert _{n}=1$. If
$u_{k}\rightarrow0$ weakly in $W_{0}^{1,n}\left(  B\right)  $, then
\[
\underset{k\rightarrow\infty}{\lim\sup}\int_{B}e^{\alpha_{n}\left\vert
u_{k}\right\vert ^{\frac{n}{n-1}}dx}\leq B\left(  1+e^{1+\frac{1}{2}%
+\ldots+\frac{1}{n-1}}\right)  .
\]

\end{theorem}

\begin{proposition}
\label{attain lemma} If $S$ can not be attained, then
\[
S\leq\frac{\omega_{n-1}}{n}\exp\left\{  \alpha_{n}A+1+\frac{1}%
{2}+\ldots+\frac{1}{n-1}\right\},
\]
where $A$ is the constant in (\ref{13}).
\end{proposition}

\begin{proof}
\ By the Lemma \ref{tend to G}, we get%
\begin{align*}
&  \bigskip\int_{%
\mathbb{R}
^{n}\backslash B_{\delta}}\left(  \left\vert \nabla u_{k}\right\vert
^{n}+\left\vert u_{k}\right\vert ^{n}\right)  dx\\
&  =c_{k}^{\frac{-n}{n-1}}\left(  \alpha\int_{%
\mathbb{R}
^{n}\backslash B_{\delta}}U_{k}^{{n}}dx+G_{\alpha}\left(  \delta\right)
\omega_{n-1}\left\vert G^{\prime}\left(  \delta\right)  \right\vert
^{n-1}\delta^{n-1}+o_{k}\left(  1\right)  \right)  \\
&  =c_{k}^{\frac{-n}{n-1}}\left(  \alpha\underset{k\rightarrow\infty}{\lim
}\left\Vert U_{k}\right\Vert _{{n}}^{{n}}-\frac{n}{\alpha_{n}}\log
\delta+A+o_{k}\left(  1\right)  +O_{\delta}\left(  1\right)  \right)  .
\end{align*}
\bigskip Setting $\bar{u}_{k}\left(  x\right)  =\left(  u_{k}\left(  x\right)
-u_{k}\left(  \delta\right)  \right)  ^{+}$. Then we have%
\begin{align}
\bigskip\bigskip\int_{B_{\delta}}\left\vert \nabla\bar{u}_{k}\right\vert
^{n}dx &  \leq\int_{B_{\delta}}\left\vert \nabla u_{k}\right\vert ^{n}%
dx=\tau_{k}:=1-\int_{%
\mathbb{R}
^{n}\backslash B_{\delta}}\left(  \left\vert \nabla u_{k}\right\vert
^{n}+\left\vert u_{k}\right\vert ^{n}\right)  dx-\int_{B_{\delta}}\left\vert
u_{k}\right\vert ^{n}dx\nonumber\\
&  =1-c_{k}^{\frac{-n}{n-1}}\left(  \alpha\underset{k\rightarrow\infty}{\lim
}\left\Vert U_{k}\right\Vert _{{n}}^{{n}}-\frac{n}{\alpha_{n}}\log
\delta+A+o_{k}\left(  1\right)  +O_{\delta}\left(  1\right)  \right)
.\label{18}%
\end{align}
When $x\in B_{Lr_{k}}$, by (\ref{18}) and Lemma \ref{tend to G 1}, we have%
\begin{align*}
&  \alpha_{k}u_{k}^{\frac{n}{n-1}}\leq\alpha_{n}\left(  1+\alpha\left\Vert
u_{k}\right\Vert _{{n}}^{{n}}\right)  ^{\frac{1}{n-1}}\left(  \bar{u}%
_{k}+u_{k}\left(  \delta\right)  \right)  ^{\frac{n}{n-1}}\\
&  \leq\alpha_{n}\left\vert \bar{u}_{k}\right\vert ^{\frac{n}{n-1}}%
+\frac{n\alpha_{n}}{n-1}\left\vert \bar{u}_{k}\right\vert ^{\frac{1}{n-1}%
}\left\vert u_{k}\left(  \delta\right)  \right\vert +\frac{\alpha_{n}\alpha
}{n-1}\left\Vert c_{k}^{\frac{1}{n-1}}u_{k}\right\Vert _{{n}}^{{n}}%
+o_{k}\left(  1\right)  \\
&  \leq\alpha_{n}\left\vert \bar{u}_{k}\right\vert ^{\frac{n}{n-1}}%
+\frac{n\alpha_{n}}{n-1}\left\vert c_{k}\right\vert ^{\frac{1}{n-1}}\left\vert
u_{k}\left(  \delta\right)  \right\vert +\frac{\alpha_{n}\alpha}{n-1}%
\underset{k\rightarrow\infty}{\lim}\left\Vert U_{k}\right\Vert _{{n}}^{{n}%
}+o_{k}\left(  1\right)  \\
&  \leq\alpha_{n}\left\vert \bar{u}_{k}\right\vert ^{\frac{n}{n-1}}%
+\frac{n\alpha_{n}}{n-1}\left\vert G_{\alpha}\left(  \delta\right)
\right\vert +\frac{\alpha_{n}\alpha}{n-1}\underset{k\rightarrow\infty}{\lim
}\left\Vert U_{k}\right\Vert _{{n}}^{{n}}+o_{k}\left(  1\right)  \\
&  =\alpha_{n}\left\vert \bar{u}_{k}\right\vert ^{\frac{n}{n-1}}-\frac{n^{2}%
}{n-1}\log\delta+\frac{n\alpha_{n}}{n-1}A+\frac{\alpha_{n}\alpha}%
{n-1}\underset{k\rightarrow\infty}{\lim}\left\Vert U_{k}\right\Vert _{{n}%
}^{{n}}+o_{k}\left(  1\right)  +o_{\delta}\left(  1\right)  \\
&  \leq\frac{\alpha_{n}\left\vert \bar{u}_{k}\right\vert ^{\frac{n}{n-1}}%
}{\tau_{k}^{\frac{1}{n-1}}}+\alpha_{n}A-\log\delta^{n}+o_{k}\left(  1\right)
+o_{\delta}\left(  1\right)
\end{align*}
Integrating the above estimates on $B_{Lr_{k}}$, we have
\begin{align*}
\int_{B_{Lr_{k}}}\left(  \exp\left\{  \alpha_{k}u_{k}^{\frac{n}{n-1}}\right\}
-1\right)  dx &  \leq\delta^{-n}\exp\left\{  \alpha_{n}A+o_{k}\left(
1\right)  \right\}  \cdot\\
&  \cdot\int_{B_{Lr_{k}}}\left(  \exp\left\{  \alpha_{k}u_{k}^{\frac{n}{n-1}%
}/\tau_{k}^{\frac{1}{n-1}}\right\}  -1\right)  dx+o_{k}\left(  1\right)  .
\end{align*}
By the Lemma \ref{C-C}, we have%
\[
\int_{B_{Lr_{k}}}\left(  \exp\left\{  \alpha_{k}u_{k}^{\frac{n}{n-1}}\right\}
-1\right)  dx\leq\frac{\omega_{n-1}}{n}\exp\left\{  \alpha_{n}A+1+\frac{1}%
{2}+\ldots+\frac{1}{n-1}\right\}  ,
\]
thanks to Lemma \ref{lemma 8}, we get
\begin{align}
\underset{k\rightarrow\infty}{\lim}\int_{%
\mathbb{R}
^{n}}\Phi\left(  \alpha_{k}u_{k}^{\frac{n}{n-1}}\right)  dx &  \leq
\underset{L\rightarrow\infty}{\lim}\underset{k\rightarrow\infty}{\lim}%
\int_{B_{Lr_{k}}}\left(  \exp\left(  \alpha_{k}u_{k}^{\frac{n}{n-1}}\right)
-1\right)  dx\nonumber\\
&  \leq\frac{\omega_{n-1}}{n}\exp\left\{  \alpha_{n}A+1+\frac{1}{2}%
+\ldots+\frac{1}{n-1}\right\}  .\label{23}%
\end{align}
Combining (\ref{23}) and Lemma \ref{attain lemma2}, the proposition is proved.
\end{proof}

In this subsection, we will construct a function sequence $\left\{
u_{\varepsilon}\right\}  \subset W^{1,n}\left(
\mathbb{R}
^{n}\right)  $ with $\left\Vert u_{\varepsilon}\right\Vert _{W^{1,n}}=1$ such that%

\[
\int_{%
\mathbb{R}
^{n}}\Phi\left(  \alpha_{n}u_{\varepsilon}^{\frac{n}{n-1}}\right)
dx>\frac{\omega _{n-1}}{n}\exp\left\{
\alpha_{n}A+1+\frac{1}{2}+\ldots+\frac{1}{n-1}\right\} .
\]
\begin{proof}
[Proof of Theorem \ref{attain}] Let
\[
u_{\varepsilon}=\left\{
\begin{array}
[c]{c}%
\frac{C-C^{\frac{-1}{n-1}}\left(  \frac{n-1}{\alpha_{n}}\log\left(
1+c_{n}\left\vert \frac{x}{\varepsilon}\right\vert ^{\frac{n}{n-1}}\right)
-B_{\varepsilon}\right)  }{\left(  1+\alpha C^{\frac{-n}{n-1}}\left\Vert
G_{\alpha}\right\Vert _{{n}}^{{n}}\right)  ^{\frac{1}{n}}}\text{
\ \ \ \ \ \ \ \ \ \ \ \ \ \ \ \ \ \ }\left\vert x\right\vert \leq
R\varepsilon,\\
\frac{G_{\alpha}\left(  \left\vert x\right\vert \right)  }{\left(  C^{\frac
{n}{n-1}}+\alpha\left\Vert G_{\alpha}\right\Vert _{{n}}^{{n}}\right)
^{\frac{1}{n}}}\text{
\ \ \ \ \ \ \ \ \ \ \ \ \ \ \ \ \ \ \ \ \ \ \ \ \ \ \ \ \ }R\varepsilon
<\left\vert x\right\vert ,
\end{array}
\right.
\]
where $c_{n}=\left(  \frac{\omega_{n-1}}{n}\right)  ^{\frac{1}{n-1}}$,
$B_{\varepsilon}$, $R$ and $c$ depending on $\varepsilon$ will also be
determined later, such that

\medskip

\bigskip i) \ $R\varepsilon\rightarrow0$, $R\rightarrow\infty$ and
$C\rightarrow\infty$, as $\varepsilon\rightarrow0;$

\medskip

ii) $\ \frac{C-\frac{n-1}{\alpha_{n}}C^{\frac{-1}{n-1}}\log\left(
1+c_{n}\left\vert R\right\vert ^{\frac{n}{n-1}}\right)  +B_{\varepsilon}%
}{\left(  1+\alpha C^{\frac{-n}{n-1}}\left\Vert G_{\alpha}\right\Vert _{{n}%
}^{{n}}\right)  ^{\frac{1}{n}}}=\frac{G_{\alpha}\left(  R\varepsilon\right)
}{\left(  C^{\frac{n}{n-1}}+\alpha\left\Vert G_{\alpha}\right\Vert _{{n}}%
^{{n}}\right)  ^{\frac{1}{n}}};$

\medskip

We can obtain the information of $B_{\varepsilon}$, $C$ and $R$ by
normalizating $u_{\varepsilon}$. By Lemma \ref{tend to G}, we have \

\ %

\begin{align*}
&  \int_{%
\mathbb{R}
^{n}\backslash B_{R\varepsilon}}\left(  \left\vert \nabla u_{\varepsilon
}\right\vert ^{n}+\left\vert u_{\varepsilon}\right\vert ^{n}\right)  dx\\
&  =\frac{1}{C^{\frac{n}{n-1}}+\alpha\left\Vert G_{\alpha}\right\Vert _{{n}%
}^{{n}}}\int_{%
\mathbb{R}
^{n}\backslash B_{R\varepsilon}}\left(  \left\vert \nabla G_{\alpha
}\right\vert ^{n}+\left\vert G_{\alpha}\right\vert ^{n}\right)  dx\\
&  =\frac{1}{C^{\frac{n}{n-1}}+\alpha\left\Vert G_{\alpha}\right\Vert _{{n}%
}^{{n}}}\left(  -G_{\alpha}\left(  R\varepsilon\right)  \int_{\partial
B_{R\varepsilon}}\left(  \left\vert \nabla G_{\alpha}\right\vert ^{n-2}%
\frac{\partial G_{\alpha}}{\partial n}\right)  dx+\alpha\int_{%
\mathbb{R}
^{n}\backslash B_{R\varepsilon}}\left\vert G_{\alpha}\right\vert ^{n}dx\right)
\\
&  =\frac{G_{\alpha}\left(  R\varepsilon\right)  \omega_{n-1}\left\vert
G^{\prime}\left(  R\varepsilon\right)  \right\vert ^{n-1}\left(
R\varepsilon\right)  ^{n-1}+\alpha\int_{%
\mathbb{R}
^{n}\backslash B_{R\varepsilon}}\left\vert G_{\alpha}\right\vert ^{n}%
dx}{C^{\frac{n}{n-1}}+\alpha\left\Vert G_{\alpha}\right\Vert _{{n}}^{{n}}},
\end{align*}
and
\begin{align*}
\int_{B_{R\varepsilon}}\left(  \left\vert \nabla u_{\varepsilon}\right\vert
^{n}\right)  dx  &  =\frac{n-1}{\alpha_{n}\left(  C^{\frac{n}{n-1}}%
+\alpha\left\Vert G_{\alpha}\right\Vert _{{n}}^{{n}}\right)  }\int_{0}%
^{c_{n}R^{\frac{n}{n-1}}}\frac{u^{n-1}}{\left(  1+u\right)  ^{n}}du\\
&  =\frac{n-1}{\alpha_{n}\left(  C^{\frac{n}{n-1}}+\alpha\left\Vert G_{\alpha
}\right\Vert _{{n}}^{{n}}\right)  }\int_{0}^{c_{n}R^{\frac{n}{n-1}}}%
\frac{\left(  \left(  1+u\right)  -1\right)  ^{n-1}}{\left(  1+u\right)  ^{n}%
}du\\
&  =\frac{n-1}{\alpha_{n}\left(  C^{\frac{n}{n-1}}+\alpha\left\Vert G_{\alpha
}\right\Vert _{{n}}^{{n}}\right)  }\left(  \underset{k=0}{\overset{n-2}{\sum}%
}\frac{C_{n-1}^{k}\left(  -1\right)  ^{n-1-k}}{n-k-1}+\right. \\
&  +\left.  \log\left(  1+c_{n}L^{\frac{n}{n-1}}\right)  +O\left(
R^{\frac{-n}{n-1}}\right)  \right)  ,
\end{align*}
using the fact that%

\[
E:=\underset{k=0}{\overset{n-2}{\sum}}\frac{C_{n-1}^{k}\left(  -1\right)
^{n-1-k}}{n-k-1}=-\left(  1+\frac{1}{2}+\frac{1}{3}+\cdots+\frac{1}%
{n-1}\right)  ,
\]
we have
\[
\int_{B_{R\varepsilon}}\left(  \left\vert \nabla u_{\varepsilon}\right\vert
^{n}\right)  dx=-\frac{n-1}{\alpha_{n}\left(  C^{\frac{n}{n-1}}+\alpha
\left\Vert G_{\alpha}\right\Vert _{{n}}^{{n}}\right)  }\left(  E-\log\left(
1+c_{n}R^{\frac{n}{n-1}}\right)  +O\left(  R^{\frac{-n}{n-1}}\right)  \right)
\]

It is easy to check that
\[
\int_{B_{R\varepsilon}}\left(  \left\vert u_{\varepsilon}\right\vert
^{n}\right)  dx=O\left(  \left(  R\varepsilon\right)  ^{n}C^{n}\right)  ,
\]
thus we get
\begin{align*}
&  \int_{%
\mathbb{R}
^{n}}\left(  \left\vert \nabla u_{\varepsilon}\right\vert ^{n}+\left\vert
u_{\varepsilon}\right\vert ^{n}\right)  dx=\frac{1}{\alpha_{n}\left(
C^{\frac{n}{n-1}}+\alpha\left\Vert G_{\alpha}\right\Vert _{{n}}^{{n}}\right)
}\left(  \left(  n-1\right)  E+\left(  n-1\right)  \log\left(  1+c_{n}%
R^{\frac{n}{n-1}}\right)  \right. \\
&  -\log\left(  R\varepsilon\right)  ^{n}+\alpha_{n}A+\alpha\alpha
_{n}\left\Vert G_{\alpha}\right\Vert _{{n}}^{{n}}+O\left( \phi\right)
\end{align*}
where
\[
\phi= \left(  R\varepsilon\right)  ^{n}C^{n}+\left(  R\varepsilon
\right)  ^{n}\log^{n}R\varepsilon+R^{\frac{-n}{n-1}}+C^{\frac{-2n}{n-1}%
}+C^{\frac{n^{2}}{n-1}}R^{n}\varepsilon^{n}  .
\]

Setting $\int_{%
\mathbb{R}
^{n}}\left(  \left\vert \nabla u_{\varepsilon}\right\vert ^{n}+\left\vert
u_{\varepsilon}\right\vert ^{n}\right)  dx=1$, we have%

\begin{align*}
&  \alpha_{n}\left(  C^{\frac{n}{n-1}}+\alpha\left\Vert G_{\alpha}\right\Vert
_{{n}}^{{n}}\right)  =\left(  n-1\right)  E+\left(  n-1\right)  \log\left(
1+c_{n}R^{\frac{n}{n-1}}\right) \\
&  -\log\left(  R\varepsilon\right)  ^{n}+\alpha_{n}A+\alpha\alpha
_{n}\left\Vert G_{\alpha}\right\Vert _{{n}}^{{n}}+O\left( \phi\right),
\end{align*}
that is
\begin{equation}
\alpha_{n}C^{\frac{n}{n-1}}=\left(  n-1\right)  E+\log\frac{\omega_{n-1}}%
{n}-\log\varepsilon^{n}+\alpha_{n}A+O\left( \phi\right). \label{20}%
\end{equation}
On the other hand, by ii) we have%

\[
C-C^{\frac{-1}{n-1}}\left(  \frac{n-1}{\alpha_{n}}\log\left(  1+c_{n}%
\left\vert R\right\vert ^{\frac{n}{n-1}}\right)  -B_{\varepsilon}\right)
=\frac{-\frac{n}{\alpha_{n}}\log R\varepsilon+A+O\left( \phi\right)}{C^{\frac{1}{n-1}}}%
\]
which implies that%

\begin{equation}
C^{\frac{n}{n-1}}=\frac{-n}{\alpha_{n}}\log\varepsilon+\log\frac{\omega_{n-1}%
}{n}-B_{\varepsilon}+A+O\left( \phi\right). \label{21}%
\end{equation}
Combining (\ref{20}) and \bigskip(\ref{21}), we have%

\begin{equation}
B_{\varepsilon}=-\frac{n-1}{\alpha_{n}}E+O\left( \phi\right)\label{22}%
\end{equation}

\bigskip Setting $R=-\log\varepsilon\,$, which satisfies $R\varepsilon
\rightarrow0$ as $\varepsilon\rightarrow0$. We can easily verify that
\begin{equation}
\left\Vert u_{\varepsilon}\right\Vert _{{n}}^{{n}}=\frac{\left\Vert G_{\alpha
}\right\Vert _{{n}}^{{n}}+O\left(  C^{\frac{n^{2}}{n-1}}R^{n}\varepsilon
^{n}\right)  +O\left(  R^{n}\varepsilon^{n}\left(  -\log\left(  R\varepsilon
\right)  ^{n}\right)  \right)  }{C^{\frac{n}{n-1}}+\alpha\left\Vert G_{\alpha
}\right\Vert _{{n}}^{{n}}}.\text{ } \label{add5}%
\end{equation}
It is well known that when $\left\vert t\right\vert <1$,%
\[
\left(  1-t\right)  ^{\frac{n}{n-1}}\geq1-\frac{n}{n-1}t
\]
and%

\[
\left(  1+t\right)  ^{-\frac{1}{n-1}}\geq1-\frac{t}{n-1}\text{.}%
\]
\ By using above inequalities and \bigskip(\ref{add5}), we have for any $x\in
B_{R\varepsilon}$,%

\begin{align*}
&  \alpha_{n}\left\vert u_{\varepsilon}\right\vert ^{\frac{n}{n-1}}\left(
1+\alpha\left\Vert u_{\varepsilon}\right\Vert _{{n}}^{{n}}\right)  ^{\frac
{1}{n-1}}\\
&  =\alpha_{n}C^{\frac{n}{n-1}}\frac{\left(  1-C^{\frac{-n}{n-1}}\left(
\frac{n-1}{\alpha_{n}}\log\left(  1+c_{n}\left\vert \frac{x}{\varepsilon
}\right\vert ^{\frac{n}{n-1}}\right)  -B_{\varepsilon}\right)  \right)
^{\frac{n}{n-1}}}{\left(  1+\alpha C^{\frac{-n}{n-1}}\left\Vert G_{\alpha
}\right\Vert _{{n}}^{{n}}\right)  ^{\frac{1}{n-1}}}\left(  1+\alpha\left\Vert
u_{\varepsilon}\right\Vert _{{n}}^{{n}}\right)  ^{\frac{1}{n-1}}\\
&  \geq\alpha_{n}C^{\frac{n}{n-1}}\left(  1-\frac{n}{n-1}C^{\frac{-n}{n-1}%
}\left(  \frac{n-1}{\alpha_{n}}\log\left(  1+c_{n}\left\vert \frac
{x}{\varepsilon}\right\vert ^{\frac{n}{n-1}}\right)  -B_{\varepsilon}\right)
\right)  \cdot\\
&  \cdot\left(  1-\alpha C^{\frac{-n}{n-1}}\left\Vert G_{\alpha}\right\Vert
_{{n}}^{{n}}\right)  ^{\frac{1}{n-1}}\left(  1+\alpha\left\Vert u_{\varepsilon
}\right\Vert _{{n}}^{{n}}\right)  ^{\frac{1}{n-1}}\\
&  \geq\alpha_{n}C^{\frac{n}{n-1}}\left(  1-\frac{n}{n-1}C^{\frac{-n}{n-1}%
}\left(  \frac{n-1}{\alpha_{n}}\log\left(  1+c_{n}\left\vert \frac
{x}{\varepsilon}\right\vert ^{\frac{n}{n-1}}\right)  -B_{\varepsilon}\right)
\right)  \cdot\\
&  \cdot\left(  1-\alpha C^{\frac{-n}{n-1}}\left\Vert G_{\alpha}\right\Vert
_{{n}}^{{n}}\right)  ^{\frac{1}{n-1}}\left(  1+\alpha\frac{\left\Vert
G_{\alpha}\right\Vert _{{n}}^{{n}}+O\left(  C^{\frac{n^{2}}{n-1}}%
R^{n}\varepsilon^{n}\right)  +O\left(  R^{n}\varepsilon^{n}\left(
-\log\left(  R\varepsilon\right)  ^{n}\right)  \right)  }{C^{\frac{n}{n-1}}%
}\right)  ^{\frac{1}{n-1}}\\
&  \geq\alpha_{n}C^{\frac{n}{n-1}}\left(  1-\frac{n}{n-1}C^{\frac{-n}{n-1}%
}\left(  \frac{n-1}{\alpha_{n}}\log\left(  1+c_{n}\left\vert \frac
{x}{\varepsilon}\right\vert ^{\frac{n}{n-1}}\right)  -B_{\varepsilon}\right)
\right)  \cdot\\
&  \cdot\left(  1-\alpha^{2}C^{\frac{-2n}{n-1}}\left\Vert G_{\alpha
}\right\Vert _{n}^{2n}+C^{\frac{-n}{n-1}}\left(  O\left(  C^{\frac{n^{2}}%
{n-1}}R^{n}\varepsilon^{n}\right)  +O\left(  R^{n}\varepsilon^{n}\left(
-\log\left(  R\varepsilon\right)  ^{n}\right)  \right)  \right)  \right)
^{\frac{1}{n-1}}\\
&  \geq\alpha_{n}C^{\frac{n}{n-1}}\left(  1-\frac{n}{n-1}C^{\frac{-n}{n-1}%
}\left(  \frac{n-1}{\alpha_{n}}\log\left(  1+c_{n}\left\vert \frac
{x}{\varepsilon}\right\vert ^{\frac{n}{n-1}}\right)  -B_{\varepsilon}\right)
\right)  \cdot\\
&  \cdot\left(  1-\frac{1}{n-1}\alpha^{2}C^{\frac{-2n}{n-1}}\left\Vert
G_{\alpha}\right\Vert _{n}^{2n}+C^{\frac{-n}{n-1}}\left(  O\left(
C^{\frac{n^{2}}{n-1}}R^{n}\varepsilon^{n}\right)  +O\left(  R^{n}%
\varepsilon^{n}\left(  -\log\left(  R\varepsilon\right)  ^{n}\right)  \right)
\right)  \right)  \\
&  \geq\alpha_{n}C^{\frac{n}{n-1}}-n\log\left(  1+c_{n}\left\vert \frac
{x}{\varepsilon}\right\vert ^{\frac{n}{n-1}}\right)  +\frac{n\alpha_{n}}%
{n-1}B_{\varepsilon}-\frac{\alpha_{n}\alpha^{2}\left\Vert G_{\alpha
}\right\Vert _{n}^{2n}}{\left(  n-1\right)  C^{\frac{n}{n-1}}}+O\left( \phi\right).
\end{align*}
By (\ref{20}) and (\ref{22}), we obtain
\begin{align*}
&  \alpha_{n}\left\vert u_{\varepsilon}\right\vert ^{\frac{n}{n-1}}\left(
1+\alpha\left\Vert u_{\varepsilon}\right\Vert _{{n}}^{{n}}\right)  ^{\frac
{1}{n-1}}\\
&  \geq-E+\log\frac{\omega_{n-1}}{n}-\log\varepsilon^{n}-n\log\left(
1+c_{n}\left\vert \frac{x}{\varepsilon}\right\vert ^{\frac{n}{n-1}}\right)  \\
&  -\frac{\alpha_{n}\alpha^{2}\left\Vert G_{\alpha}\right\Vert _{{n}}^{2{n}}%
}{\left(  n-1\right)  C^{\frac{n}{n-1}}}+\alpha_{n}A+O\left( \phi\right).
\end{align*}
Then we have%

\begin{align*}
&  \int_{B_{R\varepsilon}}\Phi\left(  \alpha_{n}\left\vert u_{\varepsilon
}\right\vert ^{\frac{n}{n-1}}\left(  1+\alpha\left\Vert u_{\varepsilon
}\right\Vert _{{n}}^{{n}}\right)  ^{\frac{1}{n-1}}\right)  dx\\
&  \geq\exp\left\{  -E+\alpha_{n}A+\log\frac{\omega_{n-1}}{n}-\log
\varepsilon^{n}-\frac{\alpha_{n}\alpha^{2}\left\Vert G_{\alpha}\right\Vert
_{{n}}^{2{n}}}{\left(  n-1\right)  C^{\frac{n}{n-1}}}+O\left( \phi\right)\right\}  \cdot\\
&  \cdot\int_{B_{R\varepsilon}}\exp\left\{  -n\log\left(  1+c_{n}\left\vert
\frac{x}{\varepsilon}\right\vert ^{\frac{n}{n-1}}\right)  \right\} \\
&  \geq c_{n}^{n-1}\varepsilon^{-n}\exp\left\{  -E+\alpha_{n}A-\frac
{\alpha_{n}\alpha^{2}\left\Vert G_{\alpha}\right\Vert _{{n}}^{2{n}}}{\left(
n-1\right)  C^{\frac{n}{n-1}}}+O\left( \phi\right)\right\}  \int_{B_{R\varepsilon}}\left(
1+c_{n}\left\vert \frac{x}{\varepsilon}\right\vert ^{\frac{n}{n-1}}\right)
^{-n}dx\\
&  \geq\frac{\left(  n-1\right)  \omega_{n-1}}{n}\exp\left\{  -E+\alpha
_{n}A-\frac{\alpha_{n}\alpha^{2}\left\Vert G_{\alpha}\right\Vert _{{n}}^{2{n}%
}}{\left(  n-1\right)  C^{\frac{n}{n-1}}}+O\left( \phi\right)\right\}  \int_{0}%
^{c_{n}R^{\frac{n}{n-1}}}\frac{u^{n-2}}{\left(  1+u\right)  ^{n}}du\\
&  \geq\frac{\left(  n-1\right)  \omega_{n-1}}{n}\exp\left\{  -E+\alpha
_{n}A-\frac{\alpha_{n}\alpha^{2}\left\Vert G_{\alpha}\right\Vert _{{n}}^{2{n}%
}}{\left(  n-1\right)  C^{\frac{n}{n-1}}}+O\left( \phi\right)\right\}  \left(  \frac{1}%
{n-1}+o\left(  R^{\frac{-n}{n-1}}\right)  \right) \\
&  \geq\frac{\omega_{n-1}\exp\left\{  -E+\alpha_{n}A\right\}  }{n}\left(
1-\frac{\alpha_{n}\alpha^{2}\left\Vert G_{\alpha}\right\Vert _{{n}}^{2{n}}%
}{\left(  n-1\right)  C^{\frac{n}{n-1}}}+O\left( \phi\right)\right)
\end{align*}

\bigskip On the other hand, we have

\begin{align*}
\int_{%
\mathbb{R}
^{n}\backslash B_{R\varepsilon}}\Phi\left(  \alpha_{n}u_{\varepsilon}%
^{\frac{n}{n-1}}\right)  dx  & \geq\frac{\alpha_{n}^{n-1}}{\left(  n-1\right)
!C^{\frac{n}{n-1}}}\int_{%
\mathbb{R}
^{n}\backslash B_{R\varepsilon}}\left\vert G_{\alpha}\right\vert ^{n}dx\\
& =\frac{\alpha_{n}^{n-1}\left\Vert G_{\alpha}\right\Vert _{n}^{n}+O\left(
R^{n}\varepsilon^{n}\left(  \log\left(  R\varepsilon\right)  ^{n}\right)
\right)  }{\left(  n-1\right)  !C^{\frac{n}{n-1}}},
\end{align*}
thus
\begin{align*}
&  \int_{%
\mathbb{R}
^{n}}\Phi\left(  \alpha_{n}\left\vert u_{\varepsilon}\right\vert ^{\frac
{n}{n-1}}\left(  1+\alpha\left\Vert u_{\varepsilon}\right\Vert _{n}%
^{n}\right)  ^{\frac{1}{n-1}}\right)  dx\\
&  \geq\frac{\omega_{n-1}}{n}\exp\left\{  -E+\alpha_{n}A\right\}  \left(
1-\frac{\alpha_{n}\alpha^{2}\left\Vert G_{\alpha}\right\Vert _{{n}}^{2{n}}%
}{\left(  n-1\right)  C^{\frac{n}{n-1}}}+O\left( \phi\right)\right)  +\frac{\alpha_{n}%
^{n-1}\left\Vert G_{\alpha}\right\Vert _{n}^{n}}{\left(  n-1\right)
!C^{\frac{n}{n-1}}}.
\end{align*}

Since $R=\log\frac{1}{\varepsilon}$, by (\ref{21}) we have $R\sim C^{\frac
{n}{n-1}}$, then we can easily verify that $\phi=o\left(  C^{\frac{-n}{n-1}%
}\right)  $. Hence when $\alpha$ small enough, we have%

\[
\int_{%
\mathbb{R}
^{n}}\Phi\left(  \alpha_{n}\left\vert u_{\varepsilon}\right\vert ^{\frac
{n}{n-1}}\left(  1+\alpha\left\Vert u_{\varepsilon}\right\Vert _{{n}}^{{n}%
}\right)  ^{\frac{1}{n-1}}\right)
dx>\frac{\omega_{n-1}}{n}\exp\left\{ -E+\alpha_{n}A\right\}.
\]
Therefore the proof of Theorem \ref{attain} is completely finished.
\end{proof}

{\bf Acknowledgement:} The main results of this paper were presented by the second author in the AMS special session on   Geometric Inequalities and Nonlinear Partial Differential Equations in Las Vegas in April, 2015.

\end{document}